\documentclass[leqno,twoside, 12pt]{amsart}
\usepackage{amssymb, latexsym, esint}
\usepackage{color}

\usepackage{amsthm}
\usepackage{amssymb,amsfonts}
\usepackage[all,arc]{xy}
\usepackage{enumerate}
\usepackage{mathrsfs}
\usepackage{amsmath}
\usepackage{verbatim}

\usepackage[left=2.1cm,right=2.1cm,top=2.9cm,bottom=2.9cm]{geometry}

\theoremstyle{plain}
\numberwithin{equation}{section} \numberwithin{figure}{section}

\newtheorem{theorem}{Theorem}[section]
\newtheorem{lemma}[theorem]{Lemma}
\newtheorem{proposition}[theorem]{Proposition}
\newtheorem{corollary}[theorem]{Corollary}
\newtheorem{definition}[theorem]{Definition}
  \usepackage{epsfig} 
\usepackage{graphicx}  \usepackage{epstopdf}
\theoremstyle{definition}
\newtheorem{remark}[theorem]{Remark}

\newcommand{\R}{\mathbb R}
\newcommand{\N}{\mathbb N}

\numberwithin{equation}{section}

\usepackage[bookmarksnumbered,colorlinks, linkcolor=blue, citecolor=red, pagebackref, bookmarks, breaklinks]{hyperref}

\begin{document}


\title{Function spaces and potential theory in the Orlicz setting}

\author{Pablo Ochoa}

\address{P. Ochoa. \newline Universidad Nacional de Cuyo, Fac. de Ingenier\'ia. CONICET. Universidad J. A. Maza\\Parque Gral. San Mart\'in 5500\\
Mendoza, Argentina.}
\email{pablo.ochoa@ingenieria.uncuyo.edu.ar}

\author{Ariel Salort}
\address{A. Salort \newline Departamento de Matem\'aticas y Ciencia de Datos, Universidad San Pablo-CEU, CEU Universities, Urbanizaci\'on Montepr\'incipe, 28660 Boadilla del Monte, Madrid, Spain. }
\email{\tt ariel.salort@ceu.es}

\parskip 3pt

\begin{abstract}
In this article, we study certain transcendental function spaces arising in potential theory within the framework of Orlicz spaces. Specifically, we generalize Bessel and Lizorkin–Triebel spaces to the nonstandard setting of Orlicz spaces. We recover classical results from potential theory, such as the fact that Bessel–Orlicz spaces of integer order coincide with Orlicz–Sobolev spaces (Calderón type theorem), and we establish inclusion results for fractional orders. Moreover, we prove a Strauss-type lemma for potential spaces. In the last sections, we  show that certain Orlicz–Lizorkin–Triebel spaces coincide with Bessel–Orlicz spaces, and we provide a useful atomic decomposition for theses spaces. 
\end{abstract}
\maketitle

\section{Introduction}
In the early 1900s, a new framework for partial differential equations emerged to transcend the limitations of classical $C^k$ regularity. Through the pioneering works of Riesz, Nikol’skii, and Besov, fractional integration became deeply integrated with harmonic analysis and potential theory. To mention only a few fundamental references, we cite \cite{AS, D, grafakos, Riesz, St, Triebel}.

The Riesz potential $I_s$ (see \cite{Riesz}) serves as a fundamental smoothing operator that formally inverts fractional powers of the Laplacian, $(-\Delta)^{-s/2}$. While it is an essential tool for capturing local regularity, it possesses a significant structural limitation when defined on the entire space $\mathbb{R}^n$ due to its lack of global integrability: 
$$
I_s = \mathcal{F}^{-1}(|\xi|^{-s}) \implies I_s(x) \approx |x|^{s-n} \text{ as } |x| \to \infty.
$$
Consequently, $I_s \notin L^1(\mathbb{R}^n)$, which prevents the operator from being bounded on $L^p(\mathbb{R}^n)$ for any $1 \le p < \infty$. \noindent This deficiency motivates the introduction of the Bessel kernel, $G_s$ by replacing the purely homogeneous symbol $|\xi|^{-s}$ with  $(1 + |\xi|^2)^{-s/2}$, 
$$
G_s = \mathcal{F}^{-1}\left((1 + |\xi|^2)^{-s/2}\right) \implies G_s(x) \approx |x|^{s-n} \text{ as } |x|\to 0 \text{ and } G_s(x) \approx e^{-|x|} \text{ as } |x|\to\infty.
$$ 
This  ensures that $G_s \in L^1(\mathbb{R}^n)$ and leads to the robust representation $u = G_s * f$, where $f = (I - \Delta)^{s/2} u$ acts as the density or fractional derivative of $u$. The set of functions  expressed as $G_s * f$ for some $f \in L^p(\mathbb{R}^n)$ is  the \emph{Bessel potential space} $H^{s, p}(\mathbb{R}^n)$. 

These spaces were introduced and studied in the early 1960s by N. Aronszajn and K. Smith \cite{ASmith}, N. Aronszajn, F. Mulla, and P. Szeptycki  \cite{AS}, and A.P. Calderón \cite{Calderon}, who, in 1961,  proved a fundamental theorem showing the equivalence between the classical Sobolev space $W^{k,p}(\mathbb{R}^n)$ for $k \in \mathbb{N}$ and $p\in (1, \infty)$ and  the Bessel potential space $H^{k,p}(\mathbb{R}^n)$. In 1970, E. Stein popularized and refined these spaces in his book \cite{St}. Later,  the  theory was integrated into the broader Lizorkin-Triebel and Besov framework, see for instance \cite{Triebel}. 

An approach to potential spaces within local Orlicz-type settings is developed in \cite{A96, A97, A10, EGO, GHN, S}, where their fundamental formulations and primary properties are established. For related results regarding norm equivalences in the nonlocal case, we refer the reader to \cite{BC}.

The primary objective of this article is to introduce a class of potential spaces defined by growth functions more general than standard power laws, namely, those characterized by a Young function $A$.  We aim to study the fundamental properties of these spaces and establish their relationship with nonlocal Orlicz-Sobolev spaces. We recall that the \emph{fractional Orlicz-Sobolev spaces} $W^{s,A}(\R^n)$ introduced in \cite{ACPS, FBS} are the natural generalizations of the classical fractional order Sobolev spaces $W^{s,p}(\R^n)$, $p>1$.  Given a Young function $A$ and a fractional parameter $s\in (0,1)$ it is defined as
$$
W^{s,A}(\R^n)=\left\{u\colon \R^n\to \R \text{ s.t.} \int_{\R^n} A(|u|)\,dx + \iint_{\R^n\times \R^n} A\left(\frac{|u(x)-u(y)|}{|x-y|^{s}} \right)\frac{dxdy}{|x-y|^n} <\infty \right\}.
$$
For properties and further research on these spaces we refer to \cite{ACPS, ACPS1,ABS, BOT, FBS}.

Motivated by these considerations, in Section \ref{sec.potential.spaces}, for any $s \in \mathbb{R}$, we define the \emph{Bessel-Orlicz potential space} $H^{s, A}(\mathbb{R}^n)$ as the set
$$
H^{s, A}(\mathbb{R}^n):=\left\lbrace u\in L^A(\mathbb{R}^n): u= G_s*f, \text{ for some }f\in L^A(\mathbb{R}^n)\right\rbrace,
$$
which we endow with the norm
$$
\|u\|_{H^{s, A}(\mathbb{R}^n)}:=\|f\|_{L^A(\mathbb{R}^n)},
$$
where $\|\cdot\|_{L^A(\mathbb{R}^n)}$ denotes the Luxemburg norm associated with the Orlicz space $L^A(\mathbb{R}^n)$.

Our first results state the relation between these spaces. In Theorems \ref{sobolev include in Bessel} and \ref{Bessel include in Sobolev}, we address the case $s=1$ and establish the equivalence 
$$
H^{1,A}(\mathbb{R}^n) = W^{1,A}(\mathbb{R}^n),
$$ 
provided that both the Young function $A$ and its conjugate $\hat{A}$ satisfy the $\Delta_2$ condition, that is, when assuming the growth control for some constants $1<p^-\leq  p^+<\infty$ such that
$$
p^- A(t) \leq tA'(t) \leq p^+ A(t).
$$
Moreover, we prove that the previous relation can be generalized  for any $m\in\N$, i.e., 
$$
H^{m,A}(\mathbb{R}^n) = W^{m,A}(\mathbb{R}^n).
$$ 
The primary challenge in this setting is that the gradient of the Bessel kernel, $\nabla G_1$, does not belong to $L^1(\mathbb{R}^n)$, which precludes a simple convolution argument. Due to this, the core of our proof lies in proving that $\nabla G_1$ defines a Calderón-Zygmund singular integral operator. Given the assumed growth conditions on $A$ and $\hat A$, we prove that this operator is bounded from $L^A(\mathbb{R}^n)$ to itself, thereby bridging the gap between the potential-based and derivative-based definitions of these spaces.

In Theorem \ref{teo.s1} we prove that for any $0<s<s'<1$, the continuous inclusion 
$$
H^{s',A}(\R^n) \subset W^{s,A}(\R^n)
$$ 
holds for any Young function $A$ regardless of assuming the $\Delta_2$ condition for either  $A$ or $\hat A$. We obtain both norm and modular relations between spaces. Conversely, for any $0<s<s'<1$, in Theorem \ref{teo.s2} we prove that the continuous embedding 
$$
W^{s',A}(\R^n)\subset H^{s,A}(\R^n)
$$ 
holds provided that $A$ satisfies the $\Delta_2$ condition.

Our proof differs from existing literature where, to the best of our knowledge, standard methods rely heavily on the homogeneity and power-like behavior of the modulars. The cornerstone of our nonlocal approach is the following continuity result for the Bessel kernel defined as  $K(z,x,t):=(G_\alpha(z-x)-G_\alpha(z-x+t))|t|^{-\gamma}$. In Proposition \ref{prop.cont}, we establish that for $\alpha>\gamma>0$, and with the measure $d\mu=|t|^{-n}dxdt$, there exists a constant $C>0$ such that:
$$
\int_{\R^n} \iint_{\R^n \times \R^n} |K(z,x,t) v(x,t) w(z)|\,d\mu \,dz  \leq C \|v\|_{L^A( d\mu)} \|w\|_{L^{\hat A}(\R^n)}
$$
holds for any $v\in L^A( d\mu)$ and $w\in L^{\hat A}(\R^n)$.  We point out that in the case of powers, a similar result was provided in \cite[Theorem 6.3]{AS}, however the proof depends decisively on the homogeneity of power functions. Hence, Proposition \ref{prop.cont} is established with a completely different method.  

Examples of Young function which can be considered in our results include
\begin{itemize}
\item Zygmund type functions: $A(t)=t^p\log^q (1+t^r)$ for $p>1$, $q\geq 0$, $r>0$;
\item Power combinations: $A(t)=t^p + t^q$, $1<p\leq q <\infty$;
\item Iterated logarithmic functions: $A(t) = t^p \log^\alpha(1 + t) \log^\beta(1 + \log(1 + t))$,  $p > 1$ y $\alpha, \beta \in \mathbb{R}$;
\item  Generalized Power-log functions: $A(t) = \int_0^{|t|} s^{p-1} \log^\alpha(1+s) \, ds$, $1 < p < \infty, \ \alpha \in \mathbb{R}$
\end{itemize}

In the particular case of power function, i.e., $A(t)=t^p$, with $p> 1$, our theorems recover the classical result for Bessel potential spaces developed in the 1960s by Aronszajn, Smith,  Calder\'on and Stein. For a detailed treatment of these classical cases, we refer the reader to Theorem 11.1 in the monograph \cite{AS} for the nonlocal case and \cite[Theorem p. 88]{Triebel} for the local case.

It is well established that radial functions exhibit enhanced regularity and decay properties compared to the general case. In Theorem \ref{teo.straus} we prove a Strauss-type inequality  asserting that for $ sp^->1$, any radial function $u \in H^{s,A}(\mathbb{R}^n)$ satisfies the decay
$$
|u(x)|\leq C \frac{|x|^{1-n}}{\hat A^{-1}(|x|^{1-n})} \|u\|_{H^{s,A}(\R^n)}, \qquad x\in\R^n,
$$
where the positive constant $C$ depends solely on the dimension $n$ and $p^\pm$. We remark that in the case $A(t)=t^p$, with $sp>1$, $s\in (0,1)$, Theorem \ref{teo.straus} recover the well-known estimate $|u(x)|\leq C|x|^{-(n-1)/p}\|u\|_{H^{s,p}(\R^n)}$, valid for any $u\in H^{s,p}_{rad}(\R^n)$. See for instance \cite{EP}. Moreover, the Strauss lemma has applications to local and non-local H\'enon-type equations (see for instance \cite{OS1} and \cite{OS2}).

We also consider a class of spaces based on the Littlewood-Paley decomposition. Let  $\Phi \in \mathcal{S}$ with Fourier transform $\hat \Phi$  a function such that $\text{supp}\,\hat \Phi \subset B(0, 1)$ and $\hat \Phi(\xi)=1$ on $B(0, 1/2)$.
For $k\in \mathbb{Z}$ denote $\Phi_k (x)=2^{nk}\Phi(2^k x)$ and $\phi_k(x)=\Phi_k(x)-\Phi_{k-1}(x)$. We define the \emph{Orlicz-Lizorkin-Triebel spaces} as follows
$$
F_s^{A,q}(\R^n):=\left\{u\in \mathcal{S}'\colon \|\Phi*u\|_{L^A(\R^n)} + \left\| \left( \sum_{k=1}^\infty |2^{s k}\phi_k* u|^q\right)^\frac1q \right\|_{L^A(\R^n)}<\infty \right\}, \quad q>1,
$$
endowed with the norm
$$
\|u\|_{F_s^{A,q}(\R^n)}=\|\Phi*u\|_{L^A(\R^n)} + \left\|\{2^{s k}\phi_k*u\}_{k=1}^\infty \right\|_{L^A(\ell^q)}.
$$ 
In Theorem \ref{teo.lt} we prove the equivalence with the Orlicz-Bessel potential spaces:
$$
F_s^{A,2}(\R^n)=H^{s,A}(\R^n)
$$
with equivalence of norms. Moreover, in Theorems \ref{atomic decomposition} and \ref{atomic decomposition.2} we provide for an atomic decomposition characterization of these spaces.
 
The article is organized as follows: in Section \ref{sec.prelim} we introduce the basic notation and preliminary definitions on Orlicz spaces and Bessel Kernels. Section \ref{sec.bessel} is devoted to prove our results related with Bessel-Orlicz potential spaces, giving in Section \ref{sec.strauss} a Strauss type lemma for potential spaces. Finally, in Section \ref{sec.lt} we deal with the Orlicz-Lizorkin-Triebel spaces.

\section{Preliminaries} \label{sec.prelim}

\subsection{Young functions}

We consider the following class of Young functions $A:[0, \infty) \to [0, \infty)$,
\begin{enumerate}
\item[(i)] $A$ is convex, increasing, and $A(0) = 0$.
\item[(ii)] $A$ is super-linear at zero and at infinite, that is 
$$
\lim_{t\rightarrow 0+} \frac{A(t)}{t}=0\quad \text{and}\quad \lim_{t\rightarrow \infty} \frac{A(t)}{t}=\infty.
$$
\end{enumerate}

The above Young functions $A$ admit the representation
$$
A(t)=\int_0^t a(\tau)\,d\tau,
$$
where the function $a$ is right-continuous for $t \geq 0$, positive for $t>0$, non-decreasing and
satisfies the conditions 
$$
a(0)=0, \qquad \lim_{t\to\infty} a(t)=\infty.
$$
We extend $A$ evenly to $\R$. Throughout the paper, we call $A$ a Young function if $A$ satisfies the above properties.  

\noindent  Associated to $A$ is  the Young function  complementary to it which is defined as follows:
\begin{equation}\label{Gcomp}
\hat{A} (t) := \sup \left\lbrace tw-A(w) \colon w>0 \right\rbrace .
\end{equation}

The following Young-type inequality holds
\begin{equation}\label{2.5}
t_1t_2 \leq A(t_1)+\hat{A} (t_2) \text{  for every } t_1,t_2 \geq 0.
\end{equation}

A Young function $A$ satisfies the $\Delta_2$ condition if there exists $C>2$ such that 
$$
A(2t)\leq 2A(t) \quad \text{ for all } t\geq 0.
$$

\noindent By \cite[Theorems 4.1 and 4.3, Chapter 1]{KR},   a Young function  $A$ and its complementary Young function $\hat A$ satisfy the $\Delta_{2}$ condition if and only if there is $1<p^-\leq p^+ <\infty$ such that
\begin{equation}\label{eq.p}
p^+\leq \frac{ta(t)}{A(t)} \leq p^+, ~~~~~\forall\, t>0.
\end{equation}
In particular, \eqref{eq.p} implies that
\begin{align} \label{potencias}
\min\{t^{p^-}, t^{p^+}\}&\leq A(t) \leq \max\{t^{p^-}, t^{p^+}\}, \quad t\geq 0.
\end{align}

\subsection{Orlicz-Sobolev spaces} Given a Young function $A$, we define the \emph{Orlicz space}
\begin{align*}
L^A(\R^n) = \{ u\colon  \R^n \to \R : \Phi_A(u)<\infty\},
\end{align*}
and the \emph{fractional Orlicz-Sobolev space} of order $s\in (0,1]$ 
$$
W^{s,A}(\R^n)=\{ u \colon \R^n \to \R : \Phi_A(u)+\Phi_{s,A}(u) < \infty\},
$$
where $\Phi_A$ and $\Phi_{s,A}$ are the modulars given by
\begin{align*}
\Phi_A(u)&:=\int_{\R^n} A(|u|)\,dx \\
\Phi_{s,A}(u)&:=
\begin{cases}
\displaystyle\iint_{\R^n\times\R^n} A(|D^s u|)\,d\mu, & \text{ if } s\in(0,1),\\
\displaystyle\int_{\R^n} A(|\nabla u|)\,dx & \text{ if } s=1,
\end{cases}
\end{align*}
where $d\mu=|x-y|^{-n}dxdy$ and $D^s u=\frac{u(x)-u(y)}{|x-y|^s}$.

The spaces $L^A(\R^n)$ and $W^{s,A}(\R^n)$ are endowed with the Luxemburg norms defined as
$$
\|u\|_{L^A(\R^n)} = \inf\left\{\lambda>0 \colon \Phi_A\left(\frac{u}{\lambda} \right)\leq 1 \right\}
$$
and
$$
\|u\|_{W^{s,A}(\R^n)} = \|u\|_{L^A(\R^n)} + [u]_{W^{s,A}(\R^n)},
$$
being $[u]_{W^{s,A}(\R^n)}$ the so-called Gagliardo-like seminorm defined as
$$
[u]_{W^{s,A}(\R^n)} = \inf\left\{ \lambda >0  \colon \Phi_{s,A}\left(\frac{u}{\lambda} \right)\leq 1\right\}.
$$ 
 The following H\"older's type inequality holds for Orlicz spaces (see \cite[Theorem 3.7.5]{kufner}).
\begin{lemma}
Let $A$ be a Young function with complementary function $\hat A$. If $u\in L^A(\R^n)$ and $v\in L^{\hat A}(\R^n)$ then $uv\in L^1(\R^n)$ and
$$
\int_{\R^n} |uv|\,dx \leq 2 \|u\|_{L^A(\R^n)} \|v\|_{L^{\hat A}(\R^n}.
$$
\end{lemma} 

\noindent In \cite[Theorem 4.1]{BC} it is proved the following inequality for convolutions in Orlicz spaces.
 \begin{lemma} \label{lema.1}
 Let $A$ be a Young function. Then
 $$
 \|u*v\|_{L^A(\R^n)} \leq \|v\|_{L^1(\R^n)}\|u\|_{L^A(\R^n)}
 $$
 for every $u\in L^A(\R^n)$ and $v\in L^1(\R^n)$. Moreover,
 $$
 \int_{\R^n}A(|u*v|)\,dx \leq\int_{\R^n} A\left(\|v\|_{L^1(\R^n)} |u|\right)\,dx.
 $$
 \end{lemma}


\subsection{Bessel kernels} \label{sec.potential.spaces}

Given $s\in \R$, we consider the Bessel kernel
$$
G_s(x)=\mathcal{F}^{-1}\left((1+|x|^2)^{-s/2}\right),
$$
where $\mathcal{F}(f)$ denotes the Fourier Transform of $f$.  Here, we allow $s\in \R$ since $\hat G_s(\xi)=(1+|\xi|^2)^{-s/2}$ can be identify with a tempered distribution for all $s$.  We also use the notation $\hat f$ to denote the Fourier transform and $\check f$ to denote the inverse Fourier transform.  According to \cite[Proposition 2]{S} and \cite[pages 13 and 27]{AS}, we recall some useful properties of $G_s$.
\begin{lemma} \label{lema.G}
The Bessel kernel satisfies the following properties:
\begin{itemize}
\item[i)] $G_s\in L^1(\mathbb{R}^n)$ for all $s>0$,
\item[ii)] $\nabla G_s\in L^1(\mathbb{R}^n)$ for all $s>1$,
\item[iii)] $G(x) = G(-x)$ for all $x\in \R^n$,
\item[iv)] $G(x)$ is decreasing.
\end{itemize}
\end{lemma}

The $L^1$-stability of the Bessel kernel $G_s$ and the behavior of its $L^1$-modulus of continuity is established in formulas 9.3 and 9.17 from \cite{AS} as follows.
\begin{lemma} \label{lema.2}
For $s\in(0,1)$, it holds that
$$
\int_{\R^n} |G_s(x+h)-G_s(x)|\,dx \leq \frac{C}{1-s} |h|^s.
$$
\end{lemma}

\subsection{Some further notation} The Schwartz class of rapidly decreasing $C^\infty$ functions on $\R^n$ is denoted by $\mathcal{S}$. The space of continuous linear functionals on $\mathcal{S}$ is called the space of tempered distributions and denoted by $\mathcal{S}'$.

\section{Bessel-Orlicz potential spaces and Orlicz-Sobolev spaces} \label{sec.bessel}

In this section, we introduce the Bessel-Orlicz potential spaces and we state some of their basic properties.

\begin{definition}Given a Young function $A$ and $s \in \R$, we define the Bessel-Orlicz potential space $H^{s, A}(\mathbb{R}^n)$ as
$$H^{s, A}(\mathbb{R}^n):=\left\lbrace u\in L^A(\mathbb{R}^n): u= G_s*f, \text{ for some }f\in L^A(\mathbb{R}^n)\right\rbrace.$$We equip $H^{s, A}(\mathbb{R}^n)$ with the norm
$$\|u\|_{H^{s, A}(\mathbb{R}^n)}:=\|f\|_{L^A(\mathbb{R}^n)}.$$
\end{definition}

The following lemma states that the definition of the space $H^{s, A}(\mathbb{R}^n)$ and its norm make sense. We concentrate on the case $s>0$.
\begin{lemma}Given a Young function $A$ and $s>0$, we have that for any $f\in L^A(\mathbb{R}^n)$, it follows that $G_s*f\in L^A(\mathbb{R}^n)$. Moreover, if $G_s*f_1=G_s*f_2$ for some $f_1, f_2\in L^A(\mathbb{R}^n)$, then $f_1=f_2$ a. e. in $\mathbb{R}^n$.
\end{lemma}

\begin{proof}
From the Young's inequality in Orlicz spaces Lemma \ref{lema.1} it follows that
\begin{equation}\label{young}
\|G_s*f\|_{L^A(\mathbb{R}^n)}\leq \|G_s\|_{L^1(\mathbb{R}^n)}\|f\|_{L^A(\mathbb{R}^n)}
\end{equation}and so $G_s*f$ is well-defined and belongs to $L^A(\mathbb{R}^n)$.

For the second part, let $\phi \in \mathcal{S}$. Then, by Fubini's theorem
\begin{equation}
\begin{split}
\int_{\mathbb{R}^n}(G_s*f_i)(x)\phi(x)\,dx = \int_{\mathbb{R}^n}\int_{\mathbb{R}^n}G_s(x-y)f_i(y)\phi(x)\,dy\,dx = \int_{\mathbb{R}^n}f_i(y)(G_s*\phi)(y)\,dy,
\end{split}
\end{equation}for $i=1, 2$. Hence,
$$ 
\int_{\mathbb{R}^n}f_1(y)(G_s*\phi)(y)\,dy = \int_{\mathbb{R}^n}f_2(y)(G_s*\phi)(y)\,dy.
$$
Since $G_s:\mathcal{S}\to \mathcal{S}$ is a bijection, by letting $\varphi= G_s*\phi$ and since $\phi$ is arbitrary, we get
$$
\int_{\mathbb{R}^n}(f_1-f_2)(x)\varphi(x)\,dx=0
$$
for all $\varphi\in \mathcal{S}$. Thus,  $f_1=f_2$ a. e. in $\mathbb{R}^n$.
\end{proof}

\begin{corollary}If $0<s_1<s_2$, then
$$H^{s_2, A}(\mathbb{R}^n)\subset H^{s_1, A}(\mathbb{R}^n) \quad\text{ and }\quad\|u\|_{H^{s_1, A}(\mathbb{R}^n)}\leq \|u\|_{H^{s_2, A}(\mathbb{R}^n)}.$$
\end{corollary}

\begin{proof}
Let $u\in H^{s_2, A}(\mathbb{R}^n)$ and write $u=G_{s_2}*f$, for some $f\in L^A(\mathbb{R}^n)$. Let $\varepsilon=s_2-s_1>0$. Then,
\begin{equation}
u=G_{s_2}*g = (G_{s_1}*G_{\varepsilon})*f = G_{s_1}*(G_\varepsilon*f).
\end{equation}
By Young's inequality \eqref{young}, $G_\varepsilon*f\in L^A(\mathbb{R}^n)$. Hence, $u\in  H^{s_1, A}(\mathbb{R}^n)$. Moreover, by Young's inequality again,
$$\|u\|_{H^{s_1, A}(\mathbb{R}^n)} = \|G_\varepsilon*f\|_{L^A(\mathbb{R}^n)}\leq \|G_\varepsilon\|_{L^1(\mathbb{R}^n)}\|f\|_{L^A(\mathbb{R}^n)}= \|G_\varepsilon\|_{L^1(\mathbb{R}^n)}\|u\|_{H^{s_2, A}(\mathbb{R}^n)}.$$
In view of the fact that $\|G_\varepsilon\|_{L^1(\mathbb{R}^n)}=1$, we conclude the proof.
\end{proof}

\subsection{The local case $s=1$} 
\begin{theorem}\label{sobolev include in Bessel} Let $A$ be a Young function such that $A$ and its conjugate $\hat A$  satisfy the $\Delta_2$-condition. Let $u\in W^{1, A}(\mathbb{R}^n)$. Then $u\in H^{1, A}(\mathbb{R}^n)$, and there is a constant $C>0$ independent of $u$ such that
$$\|u\|_{H^{1, A}(\mathbb{R}^n)}\leq C\|u\|_{W^{1, A}(\mathbb{R}^n)}.$$
\end{theorem}
\begin{proof}
Let $u\in W^{1, A}(\mathbb{R}^n)$. Then, by \cite[Corollary 1.10]{G}, there is a sequence $u_k\in C_0^\infty(\mathbb{R}^n) \subset \mathcal{S}$ such that $u_k\to u$ in $W^{1, A}(\mathbb{R}^n)$. Moreover, by the inversion formula \cite[Eq. (5.30)]{AS}, and letting $G_{-1}$ the inverse Bessel operator of $G_1$, we get
\begin{equation}\label{inversion u n}
G_{-1}u_k=G_1*u_k + \sum_{i=1}^n T_iD_iu_k,
\end{equation}where for each $i$, the operator $T_i$ is given by a singular integral (see \cite[Eq. (5.7)]{AS})
\begin{equation}\label{operator T}
T_i\varphi(x):=\lim_{\varepsilon\to 0^+}\int_{|x-y|>\varepsilon}D_iG_1(x-y)\varphi(y)\,dy, \quad \varphi\in \mathcal{S}.
\end{equation}We will extend the formula \eqref{inversion u n} to $u\in W^{1, A}(\mathbb{R}^n)$. 
Now, since $G_1\in L^1(\mathbb{R}^n)$ and $u_k\to u$ in $W^{1, A}(\mathbb{R}^n)$, we get by Young's inequality Lemma \ref{young} that $G_1*(u_k-u) \in L^A(\mathbb{R}^n)$ and 
\begin{equation}\label{eq1}
 \|G_1*(u_k-u)\|_{L^A(\mathbb{R}^n)}\to 0 \quad \text{as }n\to \infty.
\end{equation}

Regarding the operator $T_i$, we recall that $D_iG_1$ is not in $L^1(\mathbb{R}^n)$ but it is an analytic function outside the origin  decreasing exponentially near infinite.
However, we have by (9.2) in \cite{AS} that the following estimates hold
$$|D_iG_i(x-y)|\leq C|x-y|^{-n} \quad \text{and}\quad|D_{ij}G_i(x-y)|\leq C|x-y|^{-n-1}, \quad \text{for all }i, j=1, ..., n.$$Hence, by \cite[Eq. (5.7)]{AS}, \cite[Eq. (5.6)]{AS} which says that the Fourier transform of $D_iG_1$ is bounded, \cite[Prop. 5.2]{D} which gives the Hormander's condition, and the Calder\'on-Zygmund Theorem  \cite[Theorem 5.1]{D}, the operator $T_i$ is bounded from $L^p(\R^n)$ to $L^p(\R^n)$ for any $p>1$. In particular, it is bounded in $L^2(\R^n)$. Thus, $T_i$ is a Calder\'on-Zygmund singular integral operator (see for instance \cite[page 115]{HH}). Thus, by Corollary 5.4.3 in \cite{HH} and since $A$ and $\hat A$ satisfy the $\Delta_2$-condition, these operators are continuous from $L^A(\mathbb{R}^n)$ to $L^A(\mathbb{R}^n)$ and so  
\begin{equation}\label{eq2}
 \bigg\|\sum_{i=1}^n T_iD_i(u_k-u)\bigg\|_{L^A(\mathbb{R}^n)}\to 0 \quad \text{ as }k\to \infty.
\end{equation}Hence, by \eqref{eq1} and \eqref{eq2}, we obtain  that
\begin{equation}\label{conv1}
 \bigg\|G_1*(u_k-u) + \sum_{i=1}^n T_iD_i(u_k-u)\bigg\|_{L^A(\mathbb{R}^n)}\to 0,
\end{equation}as $k\to \infty$. Consequently, by \eqref{inversion u n}, 
\begin{equation}
u_k= G_1*(G_{-1}u_k) = G_1*\left(G_1*(u_k-u)+\sum_{i=1}^nT_iD_i(u_k-u)\right) + G_1*\left(G_1*u+\sum_{i=1}^nT_iD_iu\right).
\end{equation}Taking the limit as $k\to \infty$ in $L^A(\mathbb{R}^n)$ in the previous equality, recalling that $u_k\to u$ in $L^A(\mathbb{R}^n)$ and \eqref{conv1} together with Young's inequality, we get
$$u= G_1*\left(G_1*u+\sum_{i=1}^nT_iD_iu\right) \quad \text{a. e. in }\mathbb{R}^n,$$and hence 
$$G_{-1}u = G_1*u+\sum_{i=1}^nT_iD_iu.$$Let $f=G_{-1}u$.  Then, by Young's inequality and the continuity of the singular operators $T_i$ in $L^A(\mathbb{R}^n)$, we get $f\in L^A(\mathbb{R}^n)$ and so $u\in H^{1, A}(\mathbb{R}^n)$. Moreover,
$$\|u\|_{H^{1, A}(\mathbb{R}^n)} = \|f\|_{L^A(\mathbb{R}^n)}\leq \|G_1\|_{L^1(\mathbb{R^n})}\|u\|_{L^A(\mathbb{R}^n)}+C\sum_{i=1}^n\|D_iu\|_{L^A(\mathbb{R}^n)}=C\|u\|_{W^{1, A}(\mathbb{R}^n)}.$$This ends the proof. 
\end{proof}

We now prove the converse of Theorem \ref{sobolev include in Bessel}.
\begin{theorem}\label{Bessel include in Sobolev}
Let $A$ be a Young function such that $A$ and its conjugate $\hat A$  satisfy the $\Delta_2$-condition. Let $u\in H^{1, A}(\mathbb{R}^n)$. Then $u\in W^{1, A}(\mathbb{R}^n)$, and there is a constant $C>0$ independent of $u$ such that
$$\|u\|_{W^{1, A}(\mathbb{R}^n)}\leq C\|u\|_{H^{1, A}(\mathbb{R}^n)}.$$
\end{theorem}
\begin{proof}
Let $u\in H^{1, A}(\mathbb{R}^n)$. Then, there is $f\in L^A(\mathbb{R}^n)$ such that $u=G_1*f$. We will show that the distributional derivative $D_iu$ exists and belong to $L^A(\mathbb{R}^n)$ for all $i$.

 Let $f_k\in C_0^\infty(\mathbb{R}^n)$ converging to $f$ in $L^A(\mathbb{R}^n)$ and denote $u_k=G_1*f_k$. As $D_iG_1$ is a distribution (see \cite[page 27]{AS}),  the convolution $D_iG_1*f_k$ is well defined and it is given for any $\phi\in C_0^\infty(\mathbb{R}^n)$ by
\begin{equation}\label{derivative u 1}
\begin{split}
(D_iG_1*f_k)(\phi)&= D_iG_1(f_{k, R}*\phi) \qquad (f_{k, R}(y):=f_k(-y))\\& = -\int_{\mathbb{R}^n}G_1(x)D_i(f_{k, R}*\phi)(x)\,dx \\ &= -\int_{\mathbb{R}^n}\int_{\mathbb{R}^n}G_1(x)f_k(-y)D_i\phi(x-y)\,dx\,dy\\& =-\int_{\mathbb{R}^n}\int_{\mathbb{R}^n}G_1(x)f_k(y)D_i\phi(x+y)\,dx\,dy \\& =  -\int_{\mathbb{R}^n}\int_{\mathbb{R}^n}G_1(z-y)f_k(y)D_i\phi(z)\,dz\,dy\\& = -\int_{\mathbb{R}^n}(G_1*f_k)(z)D_i\phi(z)\,dz =-\int_{\mathbb{R}^n}u_k(z)D_i\phi(z)\,dz.
\end{split}
\end{equation}On the other hand, the distribution $D_iG_1$ can be seen as a singular integral (see (5.7) in \cite{AS}). So,
\begin{equation}\label{derivative u 2}
\begin{split}
(D_iG_1*f_k)(\phi)&=  D_iG_1(f_{k, R}*\phi) \\& = \lim_{\varepsilon\to 0^+}\int_{|x|>\varepsilon}D_iG_1(x)(f_{k, R}*\phi)(x)\,dx \\& = \lim_{\varepsilon\to 0^+}\int_{|x|>\varepsilon}\int_{\mathbb{R}^n}D_iG_1(x)f_k(y)\phi(x+y)\,dy\,dx\\&=\lim_{\varepsilon\to 0^+}\int_{\mathbb{R}^n}\int_{|z-y|>\varepsilon}D_iG_1(z-y)f_k(y)\phi(z)\,dy\,dz\\& = \lim_{\varepsilon\to 0^+}\int_{K}T^i_{\varepsilon}f_k(z)\phi(z)\,dz,
\end{split}
\end{equation}where $K=supp(\phi)$ and
$$T^i_{\varepsilon}f_k(z):= \int_{|z-y|>\varepsilon}D_iG_1(z-y)f_k(y)\,dy.$$By \cite[Theorem 5.1]{D} and \cite[Eq. (5.7)]{AS}, the operator
$$T_i\varphi(x):=\lim_{\varepsilon\to 0^+}\int_{|x-y|>\varepsilon}D_iG_1(x-y)\varphi(y)\,dy$$is bounded from $L^p(\R^n)$ to $L^p(\R^n)$ for any $p>1.$ Hence, by \cite[Theorem 5.14]{D}, the maximal operator
$$T_i^{*}\varphi(x):=\sup_{\varepsilon>0}|T^i_\varepsilon \varphi(x)|$$is bounded from $L^p(\R^n)$ to $L^p(\R^n)$ for any $p>1$. In particular, for each fixed $k$ we obtain
$$|T^i_{\varepsilon}f_k(z)| \leq T_i^{*}f_k(z) \in L^1(K), \quad K=\text{supp }\phi,$$for any $z\in K$ and any $\varepsilon>0$. Therefore, we may applied Lebesgue dominated convergence Theorem in the last equality of \eqref{derivative u 2} and take the limit $\varepsilon\to 0^+$ to obtain
\begin{equation*}
(D_iG_1*f_k)(\phi)=\int_{\mathbb{R}^n}T_if_k(z)\phi(z)\,dz.
\end{equation*}Combining this with \eqref{derivative u 1} we get
\begin{equation}
\int_{\mathbb{R}^n}u_k(y)D_i\phi(y)\,dy=-\int_{\mathbb{R}^n}T_if_k(z)\phi(z)\,dz.
\end{equation}Taking the limit as $k\to \infty$ and considering the continuity of the convolution with $G_1$ in $L^A(\mathbb{R}^n)$ (by Young's inequality) and the continuity of the distribution $D_iG_1$ (by Corollary 5.4.3 in \cite{HH}) give
$$D_iu=T_if \in L^A(\mathbb{R}^n)$$and also that there is a constant $C>0$ such that 
$$\|D_iu\|_{L^A(\mathbb{R}^n)}=\|T_if\|_{L^A(\mathbb{R}^n)}\leq C\|f\|_{L^A(\mathbb{R}^n)}.$$This ends the proof of the theorem.
\end{proof}

\begin{remark}
Theorems \ref{Bessel include in Sobolev} and \ref{sobolev include in Bessel} hold for $s=m$, being $m\in \mathbb{N}.$ Indeed, from \cite[Eq. (5.30)]{AS}, the representation formula takes the form
$$G_{-m}u=\sum_{l=0}^{m-1}\begin{pmatrix}
m \\ 
l
\end{pmatrix}\sum_{|j|=l}(D_jG_m)*D_ju + \sum_{|j|=m}T_jD_ju,$$where
$$T_j\varphi(x):=\lim_{\varepsilon\to 0^+}\int_{|x-y|>\varepsilon}D_jG_m(x-y)\varphi(y)\,dy.$$By (9.2) in \cite{AS}, the kernels satisfy the estimates for $|j|=m$,
$$|D_jG_m(x-y)|\leq C|x-y|^{m-n-|j|}=C|x-y|^{-n}$$and  $$|D_{j+1}G_m(x-y)|\leq C|x-y|^{m-n-|j+1|}=C|x-y|^{-n-1},$$so the above arguments apply. 
\end{remark}

\subsection{The Fractional order case}
\begin{theorem} \label{teo.s1}
Let $u\in H^{s',A}(\R^n)$ and $0<s<s'<1$. Then, $u\in W^{s, A}(\R^n)$ and there exists $C>0$ independently of $u$ such that
$$
\|u\|_{W^{s,A}(\R^n)} \leq \frac{C}{s'-s} \|u\|_{H^{s',A}(\R^n)}.
$$
Moreover, if $u=G_{s'}* f$, with $f\in L^A(\R^n)$, then the following modular inequality holds
\begin{align*}
\int_{\R^n} A(|u|)\,dx +&\int_{\R^n}\int_{\R^n}  A\left(\frac{|u(x+h)-u(x)|}{|h|^s}\right)\frac{dxdh}{|h|^n}\leq \\
&\leq  \frac{C_1}{s'-s} \int_{\R^n} A(C_2|f|)\,dx +  \int_{\R^n} A(\|G_{s'}\|_{L^1(\R^n)}|f|)\,dx,
\end{align*}

where $C_1=\frac{2n\omega n}{s}$ and $C_2=2\|G_{s'}\|_{L^1(\R^n)} + \frac{c}{1-s'}$.
\end{theorem}
\begin{proof}
Let $u\in H^{s',A}(\R^n)$ and assume that $u=G_{s'}* f$ for some  $f\in L^A(\R^n)$. Since we can write for $h\neq 0$,
\begin{align*}
\frac{u(x+h)-u(x)}{|h|^s} = \int_{\R^n} (G_{s'}(y+h)-G_{s'}(y))f(x-y)|h|^{-s}\,dy,
\end{align*}
by using Lemma \ref{lema.1} we get that
\begin{align*}
\int_{\R^n} A\left(\frac{|u(x+h)-u(x)|}{|h|^s}\right)\,dx &\leq \int_{\R^n} A\left(\int_{\R^n} |G_{s'}(y+h)-G_{s'}(y))f(x-y)||h|^{-s}\,dy \right)\,dx\\
&\leq \int_{\R^n} A\left( \int_{\R^n}|G_{s'}(y+h)-G_{s'}(y)|\,dy \cdot |f(x)| |h|^{-s}\right)\,dx\\
&= \int_{\R^n} A\left( w_{s'}(h) |f(x)| |h|^{-s}\right)dx,
\end{align*}
where we have denoted $w_{s'}(h):=\int_{\R^n}|G_{s'}(y+h)-G_{s'}(y)|\,dy $. Multiplying by $|h|^{-n}$ and integrating with respect to $h$ gives that
\begin{align} \label{ec.1}
\int_{\R^n}\int_{\R^n} A\left(\frac{|u(x+h)-u(x)|}{|h|^s}\right)\frac{dxdh}{|h|^n}
&\leq \int_{\R^n} \int_{\R^n} A\left( |w_{s'}(h)| |f(x)| |h|^{-s}\right)\frac{dxdh}{|h|^n}.
\end{align}
We will estimate the following integral:
\begin{align*}
\int_{\R^n} A\left( |w_{s'}(h)| |f(x)| |h|^{-s}\right)|h|^{-n}dh = \left(\int_{|h|\leq 1} + \int_{|h|>1} \right) A\left( |w_{s'}(h)| |f(x)| |h|^{-s}\right)|h|^{-n}dh:=I_1+ I_2.
\end{align*}

\underline{Case $|h|> 1$}. Using Lemma \ref{lema.G} we have that $|w_{s'}(h)|\leq 2\|G_{s'}\|_{L^1(\R^n)}$, so,
\begin{align*}
I_2&= \int_{|h|>1} A\left( |w_{s'}(h)| |f(x)| |h|^{-s}\right)|h|^{-n}dh\\
&\leq 
\int_{|h|>1} A\left( 2\|G_{s'}\|_{L^1(\R^n)}|f(x)| |h|^{-s}\right)|h|^{-n}dh \\
&=
\int_1^\infty  A\left( 2\|G_{s'}\|_{L^1(\R^n)} |f(x)| r^{-s}\right)r^{-1}dh\\
&\leq 
A\left(2\|G_{s'}\|_{L^1(\R^n)} |f(x)| \right) n\omega_n\int_1^\infty  r^{-1-s}dh=  \frac{n\omega_n}{s} A\left(2\|G_{s'}\|_{L^1(\R^n)} |f(x)| \right),
\end{align*}
since $r^{-s}<1$ and $A$ is convex. 

\underline{Case $|h|\leq 1$}. By Lemma \ref{lema.2}, $|w_{s'}(h)|\leq \frac{C}{1-s'}|h|^{s'}$, so, using the convexity of $A$ gives that
\begin{align*}
I_1&\leq \int_{|h|\leq 1} A\left( \tfrac{C}{1-s}|h|^{s'} |f(x)| |h|^{-s}\right)|h|^{-n}dh\\
&\leq A(\tfrac{C}{1-s'}|f(x)|) \int_{0}^1 |h|^{-n+s'-s}dh\leq \frac{n\omega_n}{s'-s}A(\tfrac{C}{1-s'}|f(x)|) 
\end{align*}
Replacing the bound for $I_1$ and $I_2$ into \eqref{ec.1} gives that
$$
\int_{\R^n}\int_{\R^n} A\left(\frac{|u(x+h)-u(x)|}{|h|^s}\right)\frac{dxdh}{|h|^n}
\leq  \frac{C_1}{s'-s} \int_{\R^n} A(C_2 |f(x)|)\,dx
$$
where $C_1=\frac{2n\omega n}{s}$ and $C_2=2\|G_{s'}\|_{L^1(\R^n)} + \frac{c}{1-s'}$.

Finally, let $\lambda := C C_2\|f\|_{L^A(\R^n)}$ for some $C>1$ such that $C>\frac{C_1}{s'-s}$. Observe that a similar computation gives that
$$
\int_{\R^n}\int_{\R^n} A\left(\frac{|u(x+h)-u(x)|}{\lambda |h|^s}\right)\frac{dxdh}{|h|^n}
\leq  \frac{C_1}{s'-s} \int_{\R^n} A\left(\frac{C_2|f(x)|}{\lambda}\right)\,dx,
$$
but, since $C>1$
$$
\frac{C_1}{s'-s} \int_{\R^n} A\left(\frac{C_2|f(x)|}{\lambda}\right)\,dx 
\leq 
\frac{C_1}{s'-s} \frac{1}{C}\int_{\R^n} A\left(\frac{|f(x)|}{ \|f\|_{L^A(\R^n)}}\right)\,dx\leq 1
$$
by definition of the Luxemburg norm. Then, using again the definition of the norm
$$
[u]_{W^{s,A}(\R^n)} \leq CC_2 \|f\|_{L^A(\R^n)}= CC_2 \|u\|_{H^{s,A}(\R^n)}.
$$

Finally, observe that by Lemma \ref{lema.1}
$$
\|u\|_{L^A(\R^n)}=\|G_s* f\|_{L^A(\R^n)}\leq  \|G_s\|_{L^1(\R^n)}\|f\|_{L^A(\R^n)} = \|G_s\|_{L^1(\R^n)}\|u\|_{H^{s,A}(\R^n)}
$$
and
$$
\int_{\R^n} A(|G_s*f|)\,dx \leq \int_{\R^n} A(\|G_s\|_{L^1(\R^n)}|f|)\,dx.
$$
This concludes the proof.
\end{proof}

\medskip

To prove the reverse inequality we need the following continuity result.
\begin{proposition} \label{prop.cont}
Let $\alpha>\gamma>0$,  $\gamma < 1$, and define $d\mu=\frac{dxdt}{|t|^n}$ and the kernel
$$
K(z,x,t):=\frac{G_\alpha(z-x)-G_\alpha(z-x+t)}{|t|^\gamma},
$$
then there exists $C>0$ such that
$$
\int_{\R^n} \iint_{\R^n \times \R^n} |K(z,x,t) v(x,t) w(z)|\,d\mu \,dz  \leq C \|v\|_{L^A( d\mu)} \|w\|_{L^{\hat A}(\R^n)}
$$
holds for any $v\in L^A( d\mu)$ and $w\in L^{\hat A}(\R^n)$.
\end{proposition}

\begin{proof}
We consider the linear operator $T$ defined on $L^A(d\mu)$ as 
$$
(Tv)(z) = \iint_{\mathbb{R}^n \times \mathbb{R}^n} K(z,x,t) v(x,t) \,d\mu.
$$
We will prove that the image of $T$ is contained in $L^A(\R^n)$ and   that $T$ is bounded from $L^A(d\mu)$ to $L^A(\mathbb{R}^n)$, that is, for any $w \in L^{\hat{A}}(\mathbb{R}^n)$
$$
I=\left|\int_{\R^n} \iint_{\R^n \times \R^n} K(z,x,t) v(x,t) w(z)\,d\mu \,dz \right| = \left| \int_{\mathbb{R}^n} (Tv)(z) w(z) \, dz \right| \leq C \|v\|_{L^A(d\mu)} \|w\|_{L^{\hat{A}}(\mathbb{R}^n)}.
$$
For this purpose, we will consider a duality approach: it is well-known that the linear operator $T\colon L^A(d\mu)\to L^A(\R^n)$ is bounded if and only if its adjoint operator $T^*\colon L^{\hat A}(\R^n)\to L^{\hat A}(d\mu)$ is bounded. The operator $T^*$ acting on $w(z)$ is defined as
$$
(T^*w)(x,t) = \frac{1}{|t|^\gamma} \int_{\mathbb{R}^n} [G_\alpha(z-x) - G_\alpha(z-x+t)] w(z) \, dz.
$$
By Lemma \ref{lema.G}, $G_\alpha$ is symmetric. Denoting $J_\alpha w(x):= (G_\alpha*w)(x)$,  we  can write
\begin{align} \label{eqt.star}
\begin{split}
(T^*w)(x,t) &= \frac{1}{|t|^\gamma} \int_{\mathbb{R}^n} [G_\alpha(\tau) - G_\alpha(\tau+t)] w(\tau+x) \, d\tau\\ 
&= \frac{(G_\alpha* w)(x) - (G_\alpha*w)(x-t)}{|t|^\gamma}=\frac{J_\alpha w(x)- J_\alpha w(x-t)}{|t|^\gamma}.
\end{split}
\end{align}
By definition of $T$ and $T^*$ and H\"older's inequality for Orlicz spaces we get formally
\begin{align*}
I &= \left|\int_{\R^n} (Tv)(z) w(z)\,dz\right| = \left|\langle Tv,w\rangle\right| = \left|\langle v,T^* w\rangle_{d\mu}\right|\\ 
&= \left|\iint_{\R^n\times \R^n} v(x,t)(T^*w)(x,t)\,d\mu \right| \leq C \|v\|_{L^A(d\mu)} \|T^*w\|_{L^{\hat A}(d\mu)}.
\end{align*}
Hence, it only remains to prove that $T^* w \in L^{\hat A}(d\mu)$ and 
\begin{equation} \label{a.probar}
\|T^* w\|_{L^{\hat A}(d\mu)} \leq C \|w\|_{L^{\hat A}(\R^n)}.
\end{equation}
We will find an upper bound of the  Luxemburg norm $\|T^* w\|_{L^{\hat A}(d\mu)}$ by analysing the modular.

\noindent \underline{$\bullet$ Case $|t|<1$ and $\alpha>1$}. In this case, due to Lemma \ref{lema.G}, $\nabla G_\alpha\in L^1(\R^n)$ since $\alpha>1$.   We start  proving that $|\nabla J_\alpha w|\in W_{loc}^{1, 1}(\R^n)$. Since $G_\alpha*w \in L^{\hat A}(\R^n)$, it is locally integrable by Holder's inequality. Next, we show that $\nabla J_\alpha w= \nabla(G_\alpha*w)$ is the regular distribution $(\nabla G_\alpha)*w $ and so it is locally integrable by Holder's and Young's inequalities.  Indeed, if $\varphi\in C_0^{\infty}(\R^n)$, then by the symmetry of $G_\alpha$, Tonelli's Theorem and Holder's inequality together with \eqref{young}, we get
\begin{equation*}
\begin{split}
\int_{\R^n}\int_{\R^n}G_\alpha(x-y)|w(y)||D_i \varphi(x)|\,dx\,dy & = \int_{\R^n}(G_\alpha*|w|)(x)|D_i\varphi(x)|\,dx \\& \leq \|G_\alpha\|_{L^1(\R^n)}\|w\|_{L^{\hat A}(\R^n)}\|D_i\varphi\|_{L^A(\R^n)}.
\end{split}
\end{equation*}Hence, Fubini's Theorem and the fact that $G_\alpha\in W^{1, 1}(\R^n)$ imply for $i=1, \ldots, n$ that
\begin{equation}
\begin{split}
D_i (G_\alpha*w)(\varphi)&=-\int_{\R^n}(G_\alpha*w)(x)D_i\varphi(x)\,dx \\ & = -\int_{\R^n}\left(\int_{\R^n}G_\alpha(x-y)D_i\varphi(x)\,dx\right)w(y)\,dy \\& = \int_{\R^n}\left(\int_{\R^n}D_iG_\alpha(x-y)\varphi(x)\,dx\right)w(y)\,dy \\ & =\int_{\R^n}(D_iG_\alpha* w)(x)\varphi(x)\,dx.
\end{split}
\end{equation}Hence, $\nabla J_\alpha w=(\nabla G_\alpha)*w \in W_{loc}^{1, 1}(\R^n)$. 
\normalcolor
By the Fundamental Theorem of Calculus (see for instance \cite[Theorem 6.9]{LL}) we have that
$$
\frac{|J_\alpha w(x) - J_\alpha w(x-t)|}{|t|^\gamma} \leq \int_0^1 |t|^{1-\gamma}  |\nabla J_\alpha w(x - st)| \, ds.
$$

We denote $\lambda:=C_1\|\nabla J_\alpha w\|_{L^{\hat A}(\R^n)}$, where $C_1>1$ is a constant to be determined. By Jensen's inequality
$$
\hat{A}\left(\frac{|(T^*w)(x,t)|}{\lambda}\right) \leq \hat{A}\left(  \int_0^1 \lambda^{-1}|t|^{1-\gamma} |\nabla J_\alpha w(x - st)| \, ds \right) 
\leq 
\int_0^1 \hat{A}\left( \lambda^{-1}|t|^{1-\gamma} |\nabla J_\alpha w(x - st)| \right) \, ds.
$$
Integrating with respect to $d\mu$, using the translation invariance  of the integral, and the fact that $1>\gamma$ give the following relation
\begin{align*}
\int_{B_1} \int_{\mathbb{R}^n} \hat{A}\left( \frac{|T^*w|}{\lambda}\right) \, dx |t|^{-n} dt 
&\leq 
\int_{B_1} |t|^{-n} \left[ \int_0^1 \left( \int_{\mathbb{R}^n} \hat{A}(\lambda^{-1}|t|^{1-\gamma} |\nabla J_\alpha w(x - st)|) \, dx \right) ds \right] dt\\
&=
\int_{B_1} |t|^{-n} \left[ \int_0^1 \left( \int_{\mathbb{R}^n} \hat{A}(\lambda^{-1}|t|^{1-\gamma} |\nabla J_\alpha w(y)|) \, dy \right) ds \right] dt\\
&=
\int_{B_1} |t|^{-n}   \left( \int_{\mathbb{R}^n} \hat{A}( \lambda^{-1}|t|^{1-\gamma} |\nabla J_\alpha w(y)|) \, dy \right) dt\\
&\leq
\int_{B_1} |t|^{-n+1-\gamma} \,dt  \left( \int_{\mathbb{R}^n} \hat{A}( \lambda^{-1}|\nabla J_\alpha w(y)|) \, dy \right)\\
&\leq
C_1 \left( \int_{\mathbb{R}^n} \hat{A}( \lambda^{-1}|\nabla J_\alpha w(y)|) \, dy \right).
\end{align*}
where we have used that $\gamma<1$ and then
$$
\int_{B_1} |t|^{-n+1-\gamma} \,dt  = n\omega_n \int_0^1 r^{-\gamma}\,dr \leq C_1,
$$
for some constant $C_1\geq 1$. This gives that
\begin{align*}
\int_{B_1} \int_{\mathbb{R}^n} \hat{A}\left( \frac{|T^*w|}{\lambda}\right) \, dx |t|^{-n} dt &\leq  C_1
 \int_{\mathbb{R}^n} \hat{A}\left( \frac{1}{C_1} \frac{|\nabla J_\alpha w(y)|}{\|\nabla J_\alpha w\|_{L^{\hat A}(\R^n)}}\right) \, dy  \\
&\leq  
 \int_{\mathbb{R}^n} \hat{A}\left( \frac{|\nabla J_\alpha w(y)|}{\|\nabla J_\alpha w\|_{L^{\hat A}(\R^n)}}\right) \, dy  =1,
\end{align*}
where we have used the definition of the Luxemburg norm.  Hence,  using Lemma \ref{lema.1}, we get that
\begin{align} \label{cota.t.1}
\begin{split}
\|T^*w\|_{L^{\hat{A}}(d\mu,  \mathbb{R}^n \times B_1  )} &\leq 
\lambda = C_1 \|\nabla J_\alpha w\|_{L^{\hat A}(\R^n)} =
C_1 \|(\nabla G_\alpha)* w\|_{L^{\hat A}(\R^n)}\\ 
&\leq 
C_1 \|\nabla G_\alpha\|_{L^1(\R^n)} \|w\|_{L^{\hat A}(\R^n)}.
\end{split}
\end{align}

\medskip
\noindent \underline{$\bullet$ Case $|t|<1$ and $\alpha<1$}.  Now, since $\alpha <1$, $\nabla G_\alpha\notin L^1(\R^n)$. We proceed by writing:
$$
(T^*w)(x,t) = \frac{1}{|t|^\gamma} \int_{\mathbb{R}^n} [G_\alpha(z-x) - G_\alpha(z-x+t)] w(z) \, dz = K_t* w(x)
$$
where
$$
K_t(x):=\frac{1}{|t|^\gamma}(G_\alpha(x)-G_\alpha(x-t)).
$$
Observe that, by Lemma \ref{lema.2}
\begin{align*}
\|K_t\|_{L^1(\R^n)} &\leq \frac{1}{|t|^\gamma}\int_{\R^n} |G_\alpha(x)-G_\alpha(x-t)|\,dx \leq C|t|^{\alpha-\gamma}.
\end{align*}
Then, by Lemma \ref{lema.1}, for $t\in B_1$
$$
\int_{\R^n} \hat A(|T^* w|)\,dx \leq \int_{\R^n}\hat A(\|K_t\|_{L^1(\R^n)} |w|)\,dx \leq |t|^{\alpha-\gamma} \int_{\R^n}\hat A(C|w|)\,dx.
$$
 Integrating with respect to $|t|^{-n}dt$, we get that
 \begin{align*}
\int_{B_1}\int_{\R^n} \hat A(|T^* w|)\,dx|t|^{-n}dt &\leq  \int_{B_1}|t|^{\alpha-\gamma-n}\,dt \int_{\R^n} \hat A(C|w|)\,dx  \\
&\leq  C_3\int_{\R^n}\hat A(C_1 |w|)\,dx  
 \end{align*}
 since $\alpha-\gamma>0$ gives that, for some $C_3\geq 1$,
 $$
 \int_{B_1}|t|^{\alpha-\gamma-n}\,dt = n\omega_n\int_0^1 r^{\alpha-\gamma-1}\,dr\leq C_3.
 $$
 Let $\lambda =C_3 C \|w\|_{L^{\hat A}(\R^n)}$, then the previous computation gives
  \begin{align*}
 \int_{B_1}\int_{\R^n} \hat A\left(\frac{|T^* w|}{\lambda}\right)\,dx|t|^{-n}dt 
 &\leq C_3 \int_{\R^n}\hat A\left(\frac{C|w|}{\lambda}\right)\,dx \leq 
  \int_{\R^n}\hat A\left(\frac{|w|}{\|w\|_{L^A(\R^n)}}\right)\,dx=1,
  \end{align*}
  where we have used the definition of the Luxemburg norm, and this implies that
\begin{equation} \label{cota.t.2}
 \|T^*w\|_{L^{\hat{A}}(d\mu, \mathbb{R}^n \times B_1 )} \leq C_3C \|w\|_{L^{\hat A}(\R^n)}.
\end{equation}

\medskip
\noindent \underline{$\bullet$ Case $|t|<1$ and $\alpha=1$}. 
 In this case $\nabla G_1\not\in L^1(\R^n)$. We use the  Calder\'on-Zygmund estimates for singular integral operators. 
 
 Observe that by \cite[(5.6)]{AS}, $|\mathcal{F}(\nabla G_1)|$ is bounded.  Moreover, by \cite[(9.2)]{AS}, 
 $$
 |D_i G_1(x)|\leq |x|^{-n}, \qquad |D_{ij} G_1(x)| \leq C |x|^{-n-1} \qquad \text{for all }i,j=1,\ldots, n.
 $$
 Therefore, in light of  \cite[Theorem 5.2]{D} and the Calder\'on-Zygmund Theorem  \cite[Theorem 5.1]{D}, the tempered distribution $\nabla G_1$ is bounded from $L^p(\R^n)$ to $L^p(\R^n)$ for any $p>1$. In particular, it is bounded in $L^2(\R^n)$ and defines a Calder\'on-Zygmund singular integral operators. 
 
 Thus, by Corollary 5.4.3 in \cite{HH} and since $A$ and $\hat A$ satisfy the $\Delta_2$-condition, this operators is continuous from $L^A(\mathbb{R}^n)$ to $L^A(\mathbb{R}^n)$ and so  
 $$
 \|(\nabla G_1)*w \|_{L^{\hat A}(\R^n)} \leq C\|w\|_{L^{\hat A}(\R^n)}.
 $$
 Hence, \eqref{cota.t.1} can be replaced by
 \begin{align} \label{cota.t.3}
 \begin{split}
 \|T^*w\|_{L^{\hat{A}}(d\mu, |t| < 1)} \leq 
 c_1 \|(\nabla G_\alpha)* w\|_{L^{\hat A}(\R^n)}
 \leq C c_1 \|w\|_{L^{\hat A}(\R^n)}.
 \end{split}
 \end{align}

 \medskip

\noindent \underline{$\bullet$ Case $|t|\geq 1$.} By triangular inequality in expression \eqref{eqt.star}
$$
\|T^*w\|_{L^{\hat{A}}(d\mu,  \mathbb{R}^n \times B_1^c  )} \leq \left\| \frac{J_\alpha w(x)}{|t|^\gamma} \right\|_{L^{\hat{A}}(d\mu,  \mathbb{R}^n \times B_1^c  )} + \left\| \frac{J_\alpha w(x-t)}{|t|^\gamma} \right\|_{L^{\hat{A}}(d\mu,  \mathbb{R}^n \times B_1^c  )}.
$$
Let us analyze the first term. We proceed as in the previous case. Denote $\lambda:= C_4 \|J_\alpha\|_{L^{\hat A}(\R^n)}$, where $C_4\geq 1$ is a  constant to determine. Then, since $|t|\geq 1$,
\begin{align*}
\int_{|t| \geq 1}\int_{\R^n} \hat{A} \left( \frac{|J_\alpha w(x)|}{ \lambda |t|^\gamma} \right) \frac{dx dt}{|t|^n} &\leq  \frac{1}{C_4}
 \int_{|t| \geq 1} \frac{1}{ |t|^{n+\gamma }} \left( \int_{\mathbb{R}^n} \hat{A} \left( \frac{|J_\alpha w(x)|}{\|J_\alpha\|_{L^{\hat A}(\R^n)}} \right) dx \right) dt\leq 1,
\end{align*}
where we have used the definition of the Luxemburg norm,  and we have chosen $C_4 =\left(\int_{|t| \geq 1} |t|^{-n-\gamma}  dt\right)^{-1}$.

\noindent  Using again the definition of the norm together with Lemma \ref{lema.1}, this gives that
$$
\left\| \frac{J_\alpha w(x)}{|t|^\gamma} \right\|_{L^{\hat{A}}(d\mu,  \mathbb{R}^n \times B_1^c  )}  \leq C_4 \|J_\alpha\|_{L^{\hat A}(\R^n)} = C_4 \|G_\alpha* w\|_{L^{\hat A}(\R^n)} \leq C_4 \|G_\alpha\|_{L^1(\R^n)} \|w\|_{L^{\hat A}(\R^n)}.
$$
The second norm term can be bounded in an analogous way, obtaining, 
\begin{equation} \label{cota.t.4}
\|T^*w\|_{L^{\hat{A}}(d\mu,  \mathbb{R}^n \times B_1^c  )} \leq
2C_4 \|G_\alpha\|_{L^1(\R^n)} \|w\|_{L^{\hat A}(\R^n)}.
\end{equation}

\medskip
Finally, from \eqref{cota.t.1}, \eqref{cota.t.2}, \eqref{cota.t.3} and \eqref{cota.t.4} we obtain \eqref{a.probar}, and the proof in concluded.
\end{proof}

\begin{theorem} \label{teo.s2}
Let $u\in W^{s',A}(\R^n)$ and let $s'>s>0$. Then, there exists $C>0$ such that
$$
\|u\|_{H^{s,A}(\R^n)} \leq C \|u\|_{W^{s',A}(\R^n)}.
$$
\end{theorem}

\begin{proof}
Let $u\in W^{s',A}(\R^n)$ and let $u_k\in C_0^\infty(\R^n)$ so that $u_k\to u$ in $ W^{s',A}(\R^n)$. The inversion formula \cite[formula (5.26)]{AS}  has the form
\begin{align}\label{def T}
\begin{split}
G_{-s}u_k&= G_{2s'-s}*u_k(z) + \int_{\R^n} \int_{\R^n} \frac{(G_{2s'-s}(z-x)-G_{2s'-s}(z-y))(u_k(x)-u_k(y))}{|x-y|^{2s'}}d\mu_{s'}(x,y)\\
&=
G_{2s'-s}*u_k(z) + \int_{\R^n} \int_{\R^n} \frac{(G_{2s'-s}(z-x)-G_{2s'-s}(z-x+t))(u_k(x)-u_k(x-t))}{|t|^{2s'}}d\mu_{s'}(x,t)\\
&=
G_{2s'-s}*u_k(z) + \int_{\R^n} \int_{\R^n} \frac{\Delta_{t}G_{2s'-s}(z-x)}{|t|^{s'}} \frac{\Delta_t u_k(x)}{|t|^{s'}}d\mu_{s'}(x,t)\\
&:=G_{2s'-s}*u_k(z) + (Tu_k)(z)
\end{split}
\end{align}
where $\Delta_t(\cdot)$ is the increment operator, and 
$$
d\mu_{s'}(x,t) = \frac{1}{C(n,s') G_{2n+2s'}(0)} \frac{G_{2n+2s'}(t)}{|t|^n}dxdt.
$$
Observe that, by \cite[Proposition 6.15]{grafakos}
\begin{align}
G_{2n+2s'}(t) \leq 
\begin{cases}
c_1(n,s')|t|^{2s'+n}\leq C_1 &\text{ if }|t|\leq 2\\
c_2(n,s')e^{-|t|/2}\leq C_2 &\text{ if }|t|\geq 2,
\end{cases}
\end{align}
then, for some positive constant $c$ depending on $s'$ and $n$ it holds that
\begin{equation} \label{d.mu}
d\mu_{s'}(x,t) \leq c s'(1-s')\frac{dxdt}{|t|^n},
\end{equation}
where we have used that $1/C(n,s')\leq s'(1-s')$, see \cite[p.274]{AS}.

We note that convolution with  $G_{2s'-s}$ is a continuous operator from  $L^A(\R^n)$ to  $L^A(\R^n)$. Indeed, by Lemma \ref{lema.1} we have that for any $v\in L^A(\R^n)$
\begin{align} \label{ineq1}
\|G_{2s'-s}*v\|_{L^A(\R^n)} \leq \|G_{2s'-s}\|_{L^1(\R^n)}\|v\|_{L^A(\R^n)}=\|v\|_{L^A(\R^n)}.
\end{align}

Regarding the operator $T$ in \eqref{def T}, Proposition \ref{prop.cont} can be applied to $T$  with $\gamma=s'$ and $\alpha=2s'-s$  since $0<s<s'<1$, which gives $2s'-s>s'$, and $\frac{\Delta_t u(x)}{|t|^{s'}} \in L^A(\R^n \times \R^n, d\mu)$. Then, $T$ is continuous from $L^{A}(\R^n \times \R^n, d\mu)$ to $L^A(\R^n)$ and so for any $v\in L^{A}(\R^n \times \R^n, d\mu)$,
\begin{align*}
\|T v\|_{L^A(\R^n)} \leq C_2  \left\| \frac{\Delta_t v(x)}{|t|^{s'}}\right\|_{L^{A}(\R^n \times \R^n, d\mu)}=C_2[v]_{W^{s',A}(\R^n)}.
\end{align*}In particular, since $u_k\to u$ in $W^{s', A}(\R^n)$, it follows
$$\|T (u_k-u)\|_{L^A(\R^n)}\leq C[u_k-u]_{W^{s',A}(\R^n)}\to 0 \quad \text{as }k\to \infty.$$

As a result, using the expression of the inversion formula in \eqref{def T}, writing
\begin{equation}
\begin{split}
u_k & =G_{s}*(G_{-s}*u_k)\\& = G_s*\left(G_{2s'-s}*(u_k-u)+T(u_k-u) \right)+ G_s*\left( G_{2s'-s}*u+T(u)\right)\\ & =  G_s*\left(G_{2s'-s}*(u_k-u)\right) + G_s*\left(T(u_k-u) \right)+ G_s*\left( G_{2s'-s}*u+T(u)\right),
\end{split}
\end{equation}and taking $k\to \infty$ in $L^A(\R^n)$, we get
$$u=G_s*\left( G_{2s'-s}*u+T(u)\right) \quad \text{a. e. in }\R^n,$$and also
$$f:=G_{-s}u = G_{2s'-s}*u+T(u) \in L^A(\R^n).$$Thus, $u\in H^{s, A}(\R^n)$. Finally,
$$\|u\|_{H^{s, A}(\R^n)}=\|f\|_{L^A(\R^n)}\leq  \|G_{2s'-s}\|_{L^1(\R^n)}\|u\|_{L^A(\R^n)}+ C[u]_{W^{s', A}(\R^n)}=C'\|u\|_{W^{s', A}(\R^n)}.$$
This concludes the proof. 
\end{proof}

\section{Strauss Lemma for potential spaces} \label{sec.strauss}

We consider the restriction of the Bessel-Orlicz potential space to radial functions. Given a Young function $A$ and $s\in \R$, we define the space
$$
H^{s,A}_{rad}(\R^n) =\{u \in H^{s,A}(\R^n)\colon u(x)=u_0(|x|)\}.
$$

\begin{lemma} \label{lema.radial}
Let $u\in H^{s,A}_{rad}(\R^n)$. Then $u=G_s*f$, where $f\in L^A(\R^n)$ is a radial function.
\end{lemma}
\begin{proof}
To establish the lemma, it suffices to show that $f$ is invariant under rotation; that is, $f \circ R = f$ for every rotation $R \in SO(n)$. Let $ (f\circ R)(\psi)$ with $\psi\in \mathcal{S}$, and define
$$
\varphi=\mathcal{F}^{-1}((1+|x|^2)^{s/2} \hat\psi).
$$
This gives that $\hat \psi=\mathcal{F}(G_s*\varphi)$, so, $\psi = G_s*\varphi$. Hence, 
\begin{align*}
 (f\circ R) (\psi) &=  f \left((G_s* \varphi)\circ R^{-1}\right) = 
 f \left((G_s\circ R^{-1})* (\varphi \circ R^{-1} )\right)
= 
 f \left( G_s* (\varphi \circ R^{-1} )\right)
\end{align*}
where we have used that $G_s$ is radial. Then, since $u\in H^{s,A}_{rad}(\R^n)$, 
\begin{align*}
( f\circ R)( \psi) &= 
 f\left( G_s* (\varphi \circ R^{-1} )\right)
= 
(f*G_s)( \varphi \circ R^{-1})
= 
( u\circ R) (\varphi)  
= 
 u( \psi).
\end{align*}
Finally, from the equality $ u( \varphi )= ( G_s*f)( \varphi)  =  f(G_s*\varphi) =  f(\varphi)$, we get the result.
\end{proof}

The following result plays a fundamental role in the proof of the Strauss Lemma. Our approach is inspired by the techniques established in \cite{EP}.

\begin{lemma} \label{lema.strauss}

Let $A$ be a Young function so that $A$ and $\hat A$ satisfy the $\Delta_2$-condition. Let $f\in L^A(\R^n)$ be a radial function.  Then for all $R>0$ and all $x$,
$$
|f*\chi_{B(0,R)}(x)|\leq C   \frac{cR^{n-1}}{\rho^{n-1}} \left\{\hat A^{-1}\left( \frac{\rho^{1-n}}{R}\right) \right\}^{-1} \|f\|_{L^A(\R^n)} , \quad  \rho:=|x|.
$$
\end{lemma}

\begin{proof}
Let $f\in L^A(\R^n)$ be a radial function and denote $f_0(r)=f(x)$. Fix $x\in \R^n$ and consider $R>0$. If $y\in B(x,R)$, taking $y=ry'$, $r=|y|$, $y'\in \mathbb{S}^{n-1}$ and $x=rx'$, $\rho=|x|$, $x'\in \mathbb{S}^{n-1}$,
\begin{align*}
f*\chi_{B(0,R)}(x)&=\int_{\rho-R}^{\rho+R} \int_{\mathbb{S}^{n-1}} f_0(r) \chi_{B(x,R)}(ry') r^{n-1} \,d\sigma(y')dr\\ 
&= 
\int_{\rho-R}^{\rho+R} f_0(r) \left(\int_{\mathbb{S}^{n-1}}  \chi_{B(x,R)}(ry') \,d \sigma(y')\right)r^{n-1}\,dr.
\end{align*}
It is not hard to see that $\chi_{B(x,R)}(ry')=\chi_{[t_0,1]}(x'y')$ with $$t_0=\frac{r^2 +   \rho^2 - R^2}{2r\rho}.$$
Moreover, if $\rho\geq R$ it is obtained that $t_0\geq -1$. Then, by \cite[Apprndix D.3]{grafakos}
\begin{align*}
\int_{\mathbb{S}^{n-1}}  \chi_{B(x,R)}(ry') \,d \sigma(y') = \int_{\mathbb{S}^{n-1}} \chi_{[t_0,1]}(x'y')\,d\sigma(y')  \leq  \int_{-1}^1 \chi_{[t_0,1]}(t)(1-t^2)^\frac{n-3}{2}\,dt.
\end{align*}
If $\rho\geq R$, we get that
\begin{align*}
f*\chi_{B(0,R)}(x) &=  \int_{\rho-R}^{\rho+R} \int_{t_0}^1 f_0(r) (1-t^2)^\frac{n-3}{2}r^{n-1}\,dtdr.
\end{align*}
Hence, by H\"older's inequality we obtain that

\begin{align*}
|f*\chi_{B(0,R)}(x)|&\leq \int_{\rho-R}^{\rho+R} |f_0(r)|(1-t_0)^\frac{n-3}{2} r^{n-1}\,dr\\ 
&\leq 
\|f_0\|_{L^A((0,\infty), r^{n-1}dr)} \|(1-t_0)^\frac{n-1}{2}\|_{L^{\hat A}((\rho-R,\rho+R), r^{n-1}dr)}.
\end{align*}

\noindent Let us bound $\|f_0\|_{L^A(\R, r^{n-1}dr)}$.
Let $C=\max\{1, n\omega_n\}^{-1}\|f\|_{L^A(\R^n)}$. Then
\begin{align*}
\int_0^\infty A\left( \frac{|f_0(r)|}{C} \right)r^{n-1}\,dr &=
\frac{1}{n\omega_n}\int_{\R^n} A\left( \frac{|f(x)|}{C} \right)\,dx\\
&\leq\int_{\R^n}
A\left(\max\{1, n\omega_n\}  \frac{|f(x)|}{C} \right)\,dx \leq 1.
\end{align*}
Then, by definition of the Luxemburg norm we get that
$$
\|f_0\|_{L^A(\R, r^{n-1}dr)}\leq C \|f\|_{L^A(\R^n)}.
$$

\noindent To bound  $\|(1-t_0)^\frac{n-1}{2}\|_{L^{\hat A}((\rho-R,\rho+R), r^{n-1}dr)}$ we proceed as follows. Let $C>0$ to determinate. Taking $u=r/\rho$, $u\in [1-R/\rho,1+R/\rho]$ and then
\begin{align*}
\int_{\rho-R}^{\rho+R}  \hat A\left(\frac{(1-t_0)^{\frac{t-1}{2}}}{C} \right)r^{n-1}\,dr 
&=
\int_{\rho-R}^{\rho+R}  \hat A\left(\frac{1}{C}\left(1-\frac{r^2 +   \rho^2 - R^2}{2r\rho}\right)^{\frac{n-1}{2}} \right)r^{n-1}\,dr\\
&=
\rho^n \int_{1-R/\rho}^{1+R/\rho}  \hat A\left( \frac{1}{C} \left(\frac{(R/\rho)^2 -(1-u)^2}{2u} \right)^\frac{n-1}{2}\right)u^{n-1}\,du\\
&\leq 
\rho^n \int_{1-R/\rho}^{1+R/\rho}  \hat A\left( \frac{1}{C}\left(\frac{2}{u} \frac{R^2}{\rho^2} \right)^\frac{n-1}{2}\right)u^{n-1}\,du\\
&\leq 
\rho^n \int_{1-R/\rho}^{1+R/\rho}  \hat A\left( \frac{1}{C} \frac{1}{u^\frac{n-1}{2}}\frac{R^{n-1}}{\rho^{n-1}} \right)u^{n-1}\,du.
\end{align*}
If $\rho\geq 2R$, $u\in[1/2,3/2]$, so
\begin{align*}
\int_{\rho-R}^{\rho+R}  \hat A\left(\frac{(1-t_0)^{\frac{t-1}{2}}}{C} \right)r^{n-1}\,dr 
&\leq
\rho^n  \hat A\left( \frac{c}{C} \frac{R^{n-1}}{\rho^{n-1}} \right) \frac{R}{\rho}\leq 1
\end{align*}
when we take $C= \frac{cR^{n-1}}{\rho^{n-1}} \left\{\hat A^{-1}\left( \frac{\rho^{1-n}}{R}\right) \right\}^{-1}$. Therefore, by definition of the Luxemburg norm
$$
\|(1-t_0)^\frac{n-1}{2}\|_{L^{\hat A}((\rho-R,\rho+R), r^{n-1}dr)} \leq  \frac{cR^{n-1}}{\rho^{n-1}} \left\{\hat A^{-1}\left( \frac{\rho^{1-n}}{R}\right) \right\}^{-1}.
$$
This yields that, when $\rho\geq 2R$,
\begin{equation} \label{f.1}
|f*\chi_{B(0,R)}(x)|\leq C \|f\|_{L^A(\R^n)}   \frac{cR^{n-1}}{\rho^{n-1}} \left\{\hat A^{-1}\left( \frac{\rho^{1-n}}{R}\right) \right\}^{-1}.
\end{equation}

If $\rho<2R$, using Example 3.6.9 of \cite{kufner}
\begin{align} \label{f.2}
\begin{split}
|f*\chi_{B(0,R)}(x)|&\leq \int_{\R^n} |f(y)||\chi_{B(0,R)}(x-y)\,dy|\\
&\leq
\|f\|_{L^A(\R^n)} \| \chi_{B(0,R)}\|_{L^{\hat A}(\R^n)}\\
&\leq
\|f\|_{L^A(\R^n)} R^n A^{-1}\left(R^{-n} \right).
\end{split}
\end{align}

Let us see that, if $\rho<2R$, then
$$
R^n A^{-1}\left(R^{-n} \right) \le \frac{CR^{n-1}}{\rho^{n-1}} \left\{\hat A^{-1}\left( \frac{\rho^{1-n}}{R}\right) \right\}^{-1}.
$$
This is equivalent to see that
$$
  \hat A^{-1}\left( \frac{\rho^{1-n}}{R}\right) \left(\frac{\rho^{1-n}}{R} \right)^{-1}  \le   \frac{C}{A^{-1}\left(R^{-n} \right)}.
$$
We use that for all $t>0$, $t\leq A^{-1}(t)\hat A^{-1}(t)\leq 2t$, giving that
$$
  \hat A^{-1}\left( \frac{\rho^{1-n}}{R}\right) \left(\frac{\rho^{1-n}}{R} \right)^{-1} \leq \frac{2}{A^{-1}(\rho^{1-n}/R)}.
$$
Let us see that
$$
\frac{2}{A^{-1}(\rho^{1-n}/R)} \leq \frac{C}{A^{-1}\left(R^{-n} \right)}. 
$$
Since we assume the $\Delta_2$ condition, this is equivalent to see that, for some $C_1\geq 2$
$$
 A^{-1}\left(C_1 R^{-n} \right) \leq A^{-1}(\rho^{1-n}/R)
$$
and applying $A$ to both sides, this equals to
$$
C_1 R^{1-n} \leq \rho^{1-n}.
$$
But, $\rho<2R$ implies that $2^{1-n} R^{1-n} < \rho^{1-n}$, so we can take $C_1=2^{1-n}$. 

This concludes the proof.
\end{proof}

\begin{theorem} \label{teo.straus}
Let $A$ be a Young function so that $A$ and $\hat A$ satisfy the $\Delta_2$-condition. Let $u\in H^{s,A}_{rad}(\R^n)$ with  $1<sp^-$. Then there exists $C>0$ depending of $n$, $p^+$ and $p^-$ such that
$$
|u(x)|\leq C \frac{|x|^{1-n}}{\hat A^{-1}(|x|^{1-n})} \|u\|_{H^{s,A}(\R^n)}.
$$\end{theorem}

\begin{proof}
Let $u\in H^{s,A}(\R^n)$ be a radial function. Then, from Lemma \ref{lema.radial} there exists a radial function  $f\in L^A(\R^n)$, $f=f_0(|x|)$ such that $u=G_s*f$.

Let $f\in L^A(\R^n)$ be a radial function and denote $f_0(r)=f(x)$. Without loss of generality, assume $f\geq 0$. Since $G_s$ is decreasing, for a fix $a>0$ 
\begin{align*}
f*G_s(x) &= \int_{\R^n} f(x-y)G_s(y)\,dy=\sum_{k\in \mathbb{Z}} \int_{2^{k-1}a \leq |y|\leq 2^k a} f(x-y)G_s(y)\,dy\\
&\leq \sum_{k\in \mathbb{Z}}  G_s(2^{k-1}a ) \int_{ |y|\leq 2^k a} f(x-y)\,dy\\
&= \sum_{k\in \mathbb{Z}}  G_s(2^{k-1}a ) \int_{\R^n} f(x-y)\chi_{B(0,2^ka)}(y)\,dy\\
&= \sum_{k\in \mathbb{Z}}  G_s(2^{k-1}a )  f*\chi_{B(0,2^ka)}.
\end{align*}

From Lemma \ref{lema.strauss}, for a fix $a>0$, with $R=2^{k}a$, $\rho=|x|$, it  yields 
\begin{align*}
|G_s*f|&\leq \sum_{k\in \mathbb{Z}} G_s(2^{k-1}a) |g* \chi_{B(0,2^k a)}(x)|\\
&\leq C \|f\|_{L^A(\R^n)} |x|^{1-n} \sum_{k\in \mathbb{Z}} G_s(2^{k-1}a)   (2^k a)^{n-1} \left\{\hat A^{-1}\left( \frac{|x|^{1-n}}{2^k a}\right) \right\}^{-1}.
\end{align*}
If $r_k=2^{k-1}a$, then $\Delta r_k = r_{k+1}-r_k=2^{k-1}a$, so
\begin{align*}
\sum_{k\in \mathbb{Z}} G_s(2^{k-1}a)   (2^k a)^{n-1} \left\{\hat A^{-1}\left( \frac{|x|^{1-n}}{2^k a}\right) \right\}^{-1}&=
2^{n-1} \sum_{k\in \mathbb{Z}} G_s(r_k)   (r_k)^{n-1}  \left\{\hat A^{-1}\left( \frac{|x|^{1-n}}{2r_k}\right)  \right\}^{-1}\frac{\Delta r_k}{r_k}.
\end{align*}
This is the Riemann sum for the integral
$$
2^{n-1}\int_0^\infty G_s(r) r^{n-2} \left\{\hat A^{-1}\left( \frac{|x|^{1-n}}{2r}\right)  \right\}^{-1} \, dr.
$$
Denote by $\hat p^+>1$ the number such that
$$
\frac{t (\hat A(t))'}{\hat A(t)}\leq \hat p^+ \quad \text{ for all }t\geq 0.
$$
Then, it follows that
$$
\left\{\hat A^{-1}\left( \frac{|x|^{1-n}}{2r_k}\right)  \right\}^{-1}\leq  (2r_k)^\frac{1}{\hat p^+} \left\{\hat A^{-1}\left( |x|^{1-n}\right)  \right\}^{-1},
$$
and  letting $a\to 0$, we get
\begin{align*}
|G_s*f|&\leq C \|f\|_{L^A(\R^n)} \frac{|x|^{1-n}}{\hat A^{-1}(|x|^{1-n})} \int_0^\infty G_s(r) r^{n-2}  (2r)^\frac{1}{\hat p^+} \, dr\\
&= C \|f\|_{L^A(\R^n)} \frac{|x|^{1-n}}{\hat A^{-1}(|x|^{1-n})}  \int_{\R^n} G_s(x) |x|^{-1+\frac{1}{\hat p^+}} \, dx\\
&\leq  C \|f\|_{L^A(\R^n)} \frac{|x|^{1-n}}{\hat A^{-1}(|x|^{1-n})},  
\end{align*}
since the following relation holds $\frac{1}{\hat p^+}=1-\frac{1}{p^-}$, and hence
\begin{align*}
\frac{1}{n\omega_n}\int_{\R^n} G_s(x) |x|^{-1+\frac{1}{\hat p^+}} \, dx&=
\int_0^2 G_s(r) r^{n-1-1+\frac{1}{\hat p^+}} \, dr + 
\int_2^\infty G_s(r) r^{n-1-1+\frac{1}{\hat p^+}} \, dr\\
&\leq
C\int_0^2 r^{s-n}  r^{n-1-1+\frac{1}{\hat p^+}}\,dr +C \int_2^\infty e^{-R/2} r^{n-1-1+\frac{1}{\hat p^+}} \, dr\\
&=
C\int_0^2   r^{s-1-\frac{1}{p^-}}\,dr +C \int_2^\infty e^{-R/2} r^{n-1-1+\frac{1}{\hat p^+}} \, dr\leq C,
\end{align*}
where we have used that $s>1/p^-$.  The proof in now completed.
\end{proof}

\section{Orlicz-Lizorkin-Triebel spaces} \label{sec.lt}

In this section, we generalized the Lizorkin-Triebel spaces to the framework of Orlicz spaces.

Let $\Phi \in \mathcal{S}$ and let $\hat \Phi$ be its Fourier transform. Suppose that
$$\text{supp}\,\hat \Phi \subset B(0, 1), \quad \text{and}\quad \hat \Phi(\xi)=1 \quad \text{on} \quad B(0, 1/2).$$

 Set for $k\in \mathbb{Z}$,
$$
\Phi_k (x)=2^{nk}\Phi(2^k x) \quad \text{ so that } \quad \hat \Phi_k(\xi)=\hat\Phi(2^{-k}\xi),
$$
and define $\phi_k(x)=\Phi_k(x)-\Phi_{k-1}(x)$. Observe that

$$\text{supp}\,\hat \phi_k \subset B(0, 2^{k})\setminus B(0, 2^{k-2})$$and moreover
$$\hat \Phi (\xi) +\sum_{k=1}^\infty \hat \phi(\xi)=1 \quad \text{for all }\xi.$$

 Given a Young function $A$, $q>1$, and $s\in \R$, we define the Orlicz-Lizorkin-Triebel spaces as follows
$$
F_s^{A,q}(\R^n):=\left\{u\in \mathcal{S}'\colon \|\Phi*u\|_{L^A(\R^n)} + \left\| \left( \sum_{k=1}^\infty |2^{s k}\phi_k* u|^q\right)^{1/q} \right\|_{L^A(\R^n)}<\infty \right\}
$$
 endowed with the norm
$$
\|u\|_{F_s^{A,q}(\R^n)}=\|\Phi*u\|_{L^A(\R^n)} + \left\|\{2^{s k}\phi_k*u\}_{k=1}^\infty \right\|_{L^A(\ell^q)},
$$ 
where
$$
\|\{f_k\}_{k=1}^\infty\|_{L^A(\ell^q)} =\inf\left\{\lambda>0 \colon \int_{\R^n} A\left(\frac{\|f(x)\|_q}{\lambda} \right)\,dx\leq 1 \right\}
$$
being
$$
\|f(x)\|_q = \left(\sum_{k=1}^\infty |f_k(x)|^q\right)^\frac1q.
$$

Observe that in the definition of $F_s^{A, q}(\R^n)$ we ask whether $\Phi*u$ and  $\{2^{s k}\phi_k*u\}_{k=1}^\infty $ are regular distributions belonging to $L^A(\R^n)$ and $L^A(\ell^q)$, respectively. Also, we point out that as in the classical setting, we may define $F_s^{A,q}(\R^n)$ via Fourier transform as
\begin{equation}\label{alternative LT}
F_s^{A,q}(\R^n):=\left\{u\in \mathcal{S}'\colon \left\| \left( \sum_{k=0}^\infty |2^{s k}\mathcal{F}^{-1}(\hat \phi_k \hat u)|^q\right)^\frac1q \right\|_{L^A(\R^n)}<\infty \right\}.
\end{equation}


\subsection{A Lizorkin-Triebel characterization of potential spaces}

\begin{theorem} \label{teo.lt}
Assume that $A$ and $\hat A$ satisfy the $\Delta_2$-condition. Then, the Orlicz-Lizorkin-Triebel space with $q=2$ coincides with the Orlicz-Bessel potential space:
$$
F_s^{A,2}(\R^n)=H^{s,A}(\R^n)
$$
and there are constants $c_1$ and $c_2$ such that
$$
c_1 \|u\|_{H^{s,A}(\R^n)}\leq \|u\|_{F^{A,2}_s(\R^n)} \leq c_1 \|u\|_{H^{s,A}(\R^n)}.
$$
\end{theorem}

\begin{proof}
Let $u\in H^{s, A}(\R^n)$. By density, we may assume that $\hat u$ has compact support. Write $u=G_s*f$, with $f\in L^A(\R^n)$. Observe that $\hat u=\hat G_s \hat f$ and so
\begin{equation}\label{f hat}
\hat f(\xi) = (1+|\xi|^2)^{s/2}\hat u.
\end{equation}Observe that $f\in \mathcal{S}$.

Let the Rademacher functions $\left\lbrace r_i \right\rbrace_{i=0}^\infty$ given by
\begin{alignat*}{2}
&r_0(t)   = 1          & \quad & \text{for } 0 < t \leq \frac{1}{2}, \\
&r_0(t)   = -1         & \quad & \text{for } \frac{1}{2} < t \leq 1, \\
&r_0(t+1) = r_0(t)     & \quad & \text{for } t \in \mathbb{R} ,      \\
&r_i(t)   = r_0(2^i t) & \quad & \text{for } i \in \mathbb{N}.
\end{alignat*}
It is known that Rademacher functions form an orthonormal system in $L^2(0, 1)$ (see for instance \cite{AH}) and satisfies the Khinchin inequality (\cite[Lemma 4.2.3]{AH}):
\begin{equation}\label{K ineq}
\bigg\|\sum_{i=0}^ma_ir_i(t)\bigg\|_{L^2(0,1)}=\left(\int_0^1\bigg|\sum_{i=0}^ma_ir_i(t) \bigg|^2\,dt\right)^{1/2}= \left(\sum_{i=0}^m|a_i|^2\right)^{1/2}\leq \sqrt{3}\bigg\|\sum_{i=0}^ma_ir_i(t)\bigg\|_{L^1(0,1)}
\end{equation}for any constants $a_i$. Define
$$m(\xi):=\dfrac{1}{\hat \Phi^2(\xi) + \sum_{i=1}^\infty\hat \phi_i^2(\xi)}.$$Then, by \cite[Lemma 4.2.5]{AH}, $m$ can be written, for any $t\in [0, 1]$, as
$$m(\xi):=\dfrac{r_0(t)\hat \Phi(\xi)+\sum_{i=1}^\infty 2^{is}r_i(t)\hat \phi_i(\xi)}{(1+|\xi|^2)^{s/2}},$$
and satisfies the assumptions of \cite[Theorem 4.2.4]{AH}, namely, $m\in L^{\infty}(\R^n)$ and for all $R>0$ and all multiindex $\sigma$ with $|\sigma|\leq k$, with $k$ is the least integer so that $k>n/2$, there holds
\begin{equation*}
\begin{split}
\dfrac{1}{R^n}\int_{\left\lbrace 1/2  R\leq |\xi|\leq 2R \right\rbrace }|R^{|\sigma|}D^\sigma m(\xi)|^2\,d\xi \leq B
\end{split}
\end{equation*}for some $B>0$. Moreover, defining $m_i:=m\hat \phi_i$, $k_i$ the inverse Fourier transform of $m_i$ and setting $M_i:=\sum_{j=-i}^{j=i}m_j$ and $K_i:=\sum_{j=-i}^{j=i}k_j$. Then according to the proof of Theorem 4.2.4 in \cite{AH}, $K_i$ satisfies the assumptions of the Calder\'on-Zygmund Theorem \cite[Theorem 5.1 and Theorem 7.11]{D}. Therefore, by \cite[Corollary 5.4.3]{HH}, the mapping
$$f\to K_i*f=\mathcal{F}^{-1}(M_i\mathcal{F}(f)) \quad \text{is continuous from }L^A(\R^n) \text{ to }L^A(\R^n)$$with a constant independent of $i$. Now, since $\|M_i\|_{L^\infty(\R^n)}$ is uniformly bounded and $M_i(\xi)\to m(\xi)$ for a.e. $\xi \in \R^n$, we have that
$$M_i \to m \text{ in }\mathcal{S}' \text{ as }i\to \infty.$$Since the inverse Fourier transform is an isomorphism in $S'$, we have
$$K_i \to T\text{ in $\mathcal{S}'$ as }i\to \infty$$for some distribution $T$ so that $\hat T=m$. The norm in $L^A(\R^n)$ of the mapping $f \to T* f$ with $f\in \mathcal{S}$ is
\begin{equation*}
\begin{split}
\sup_{\|f\|_{L^A(\R^n)}\leq 1}\|T(f)\|_{L^A(\R^n)} &= \sup\left\lbrace \left|\int_{\R^n}(T*f)g \,dx\right|: f, g\in S,\|f\|_{L^A(\R^n)}\leq 1, \|g\|_{L^{\hat A}(\R)}\leq 1 \right\rbrace\\& =\sup\left\lbrace \left|T(f*g)\right|: g\in S,\|f\|_{L^A(\R^n)}\leq 1, \|g\|_{L^{\hat A}(\R)}\leq 1 \right\rbrace.
\end{split}
\end{equation*}Now, since
$$T(f*g)=\lim_{i\to \infty}K_i(f*g)\leq C\|f\|_{L^A(\R^n)}\|g\|_{L^{\hat A}(\R)}\leq C,$$we conclude, by density, that convolution with $T$ is continuous from $L^A(\R^n)$ to $L^A(\R^n)$. In particular,
\begin{equation}\label{LT1}
\|\mathcal{F}^{-1}(m\hat{f})\|_{L^A(\R^n)}\leq C\|f\|_{L^A(\R^n)}.
\end{equation}Now, turning back to $f$ so that \eqref{f.1} holds, we have
\begin{equation}\label{LT2}
\begin{split}
\|\mathcal{F}^{-1}(m\hat{f})\|_{L^A(\R^n)}&=\bigg\|\mathcal{F}^{-1}\left(\left(r_0(t)\hat \Phi(\xi)+\sum_{i=1}^\infty 2^{is}r_i(t)\hat \phi_i(\xi)\right)\hat{u}\right)\bigg\|_{L^A(\R^n)}\\ & = \bigg\| r_0(t)\Phi*u+\sum_{i=1}^\infty 2^{is}r_i(t)\phi_i*u\bigg\|_{L^A(\R^n)}.
\end{split}
\end{equation}Observe that by the assumptions on $\hat u$, the above sums are indeed finite. Hence, from \eqref{LT1} and \eqref{LT2},
\begin{equation}\label{ineq norms}
\bigg\| r_0(t)\Phi*u\sum_{i=1}^\infty 2^{is}r_i(t)\phi_i*u\bigg\|_{L^A(\R^n)} \leq C\|f\|_{L^A(\R^n)} = C\|u\|_{H^{s, A}(\R^n)}.
\end{equation}Integration in $t$ gives
\begin{equation}\label{ineq 34}
\int_{0}^1\bigg\| r_0(t)\Phi*u+\sum_{i=1}^\infty 2^{is}r_i(t)\phi_i*u\bigg\|_{L^A(\R^n)}\,dt \leq  C\|u\|_{H^{s, A}(\R^n)}.
\end{equation}Now, applying Khinchin inequality \eqref{K ineq} and \eqref{ineq 34}
\begin{equation}\label{end first part}
\begin{split}
\left\|\left(|\Phi*u|^2+\sum_{i=1}^\infty 2^{2is}|\phi_i*u|^2\right)^\frac12\right\|_{L^A(\R^n)} &\leq C\left\| \int_{0}^1\left(r_0(t)\Phi*u+\sum_{i=1}^\infty 2^{is}r_i(t)\phi_i*u\right)\,dt\right\|_{L^A(\R^n)}\\ & \leq C\int_0^1\bigg\|r_0(t)\Phi*u+\sum_{i=1}^\infty 2^{is}r_i(t)\phi_i*u\bigg\|_{L^A(\R^n)}\,dt \leq  C\|u\|_{H^{s, A}(\R^n)}.
\end{split}
\end{equation}We point out that the inequality
$$\bigg\| \int_{0}^1 g(t, \cdot)\,dt\bigg\|_{L^A(\R^n)}\,dt \leq C\int_0^1\bigg\|g(t, \cdot)\bigg\|_{L^A(\R^n)}\,dt$$follows by the dual characterization of the Orlicz norm and Holder's inequality. Indeed,
\begin{equation}
\begin{split}\left\| \int_{0}^1 g(t, \cdot)\,dt\right\|_{L^A(\R^n)}\,dt&=\sup\left\lbrace \int_{\R^n}\bigg|\left( \int_{0}^1 g(t, x)\,dt\right)h(x) \bigg|\,dx: \|h\|_{L^{\hat A}(\R^n)}\leq 1\right\rbrace\\& \leq \sup\left\lbrace \int_{0}^1\int_{\R^n} |g(t, x)||h(x) |\,dx\,dt: \|h\|_{L^{\hat A}(\R^n)}\leq 1\right\rbrace \\& \leq C\sup\left\lbrace\int_0^1\left\|g(t, \cdot)\right\|_{L^A(\R^n)}\cdot\|h\|_{L^{\hat A}(\R^n)}\,dt: \|h\|_{L^{\hat A}(\R^n)}\leq 1\right\rbrace \\& \leq C\int_0^1\left\|g(t, \cdot)\right\|_{L^A(\R^n)}\,dt.
\end{split}
\end{equation}
Observe that \eqref{end first part} is enough to prove the inclusion $H^{s, A}(\R^n)\subset F_s^{A, 2}(\R^n)$, since
\begin{equation}\label{final first part}
\begin{split}
\|u\|_{F_s^{A,2}(\R^n)} &= \|\Phi*u\|_{L^A(\R^n)}+\|\left\lbrace 2^{si}\phi_i*u\right\rbrace\|_{L^A(\ell^2)}\\ &= \|(|\Phi*u|^2)^{1/2}\|_{L^A(\R^n)}+\bigg\|\left(\sum_{i=1}^\infty 2^{2is}|\phi_i*u|^2\right)^{1/2} \bigg\|_{L^A(\R^n)}\\ & \leq 2\bigg\|\left(|\Phi*u|^2+\sum_{i=1}^\infty 2^{2is}|\phi_i*u|^2\right)^{1/2}\bigg\|_{L^A(\R^n)} \leq  C\|u\|_{H^{s, A}(\R^n)}.
\end{split}
\end{equation}

To prove the other inclusion, let
$$h:=\mathcal{F}^{-1}\left((\hat \Phi^2 + \sum_{i=1}^\infty \hat \phi_i^2)\hat f \right), \quad \text{that is }f=\mathcal{F}^{-1}(m\mathcal{F}(h)).$$Hence, for some $C>0$, using \eqref{LT1} we get
\begin{equation}\label{norm f and h}
\|f\|_{L^A(\R^n)}\leq C\|h\|_{L^A(\R)}. 
\end{equation}Let now $g\in \mathcal{S}$ such that $\|g\|_{L^{\hat A}(\R^n)}=1$, with compact support, and
\begin{equation}\label{second cond g}
\int_{\R^n}gh\,dx \geq \dfrac{1}{2}\|h\|_{L^A(\R^n)}.
\end{equation}Moreover, we may write $g=G_s*v$, where
$$\hat v(\xi)=(1+|\xi|^2)^{s/2}\hat g(\xi).$$
By using \eqref{norm f and h}, \eqref{second cond g} we obtain that, 
\begin{align*}
\|u\|_{H^{s, A}(\R^n)}&=\|f\|_{L^A(\R^n)} \leq C\int_{\R^n}gh\,dx = C\int_{\R^n}\hat g \hat h \,dx\\& =		C\int_{\R^n}\hat f \hat g \left(\hat \Phi^2 + \sum_{i=1}^\infty \hat \phi_i^2\right)\,d\xi\\& = C\int_{\R^n}\left((\hat u \hat \Phi)(\hat v \hat \Phi)+\sum_{i=1}^\infty (2^{is}\hat u \hat \phi_i)(2^{-is}\hat v \hat \phi_i)\right)\,d\xi.
\end{align*}
In the previous expression the constant $C$ may change from line to line. Applying Plancherel's formula and Cauchy's inequalities, we obtain
\begin{equation*}
\begin{split}
\|u\|_{H^{s, A}(\R^n)}& \leq C\int_{\R^n}\left((\Phi * u)(\Phi * v)+\sum_{i=1}^\infty (2^{is}\phi_i* u )(2^{-is}\phi_i* v)\right)dx\\ & \leq C\int_{\R^n}\left((\Phi * u)(\Phi * v) +\left(\sum_{i=1}^\infty 2^{2si}|\phi_i*u|^2 \right)^{1/2}\left(\sum_{i=1}^\infty 2^{-2si}|\phi_i*v|^2 \right)^{1/2}\right)dx
.\end{split}
\end{equation*}
Finally using Holder's inequality in the previous expression yields
\begin{equation*}
\begin{split}
\|u\|_{H^{s, A}(\R^n)}& \leq  C\bigg\{\|\Phi*u\|_{L^A(\R^n)}\|\Phi*v\|_{L^{\hat A}(\R^n)}\\&\qquad+\left\|\left( \sum_{i=1}^\infty 2^{2si}|\phi_i*u|^2 \right)^{1/2}\right\|_{L^A(\R^n)}\left\|\left( \sum_{i=1}^\infty 2^{-2si}|\phi_i*v|^2 \right)^\frac12\right\|_{L^{\hat A}(\R^n)}\bigg\}\\& \leq C\bigg\{\|(|\Phi*u|^2)^\frac12\|_{L^A(\R^n)}+\bigg\|\left(\sum_{i=1}^\infty 2^{2is}|\phi_i*u|^2\right)^\frac12 \bigg\|_{L^A(\R^n)} \bigg\}\\& \qquad \times \bigg\{\|(|\Phi*v|^2)^{1/2}\|_{L^{\hat A}(\R^n)}+\bigg\|\left(\sum_{i=1}^\infty 2^{-2is}|\phi_i*v|^2\right)^\frac12 \bigg\|_{L^{\hat A}(\R^n)} \bigg\}\\& \leq C\|u\|_{F_s^{A, 2}(\R^n)}\|v\|_{H^{-s, \hat A}(\R^n)},
\end{split}
\end{equation*}
where we have used \eqref{final first part} for $u$ and $v$.  But
$$\|v\|_{H^{-s, \hat A}(\R^n)}=\|g\|_{L^{\hat A}(\R^n)}=1.$$This concludes the proof of the theorem.
\end{proof}

\subsection{An atomic decomposition in Orlicz-Lizorkin-Triebel spaces}

We start with an elementary lemma which provides a representation for any $u\in F_s^{A,q}(\R^n)$.
\begin{lemma}Let $s>0$. Any $u\in F_s^{A,q}(\R^n)$ admits the decomposition
\begin{equation}\label{representation}
u=\Phi*u+\sum_{k=1}^\infty\phi_k*u.
\end{equation}
\end{lemma}

\begin{proof}
Let $u\in F_s^{A,q}(\R^n)$. First, we show that the series 
$$\sum_{k=1}^\infty \phi_k*u$$converges in $\mathcal{S}'$. Indeed, let $\varphi\in \mathcal{S}$, we need to prove that the limit
\begin{equation}\label{desired conv}
\begin{split}
\lim_{N\to \infty}\left(\sum_{k=1}^N \phi_k*u\right)(\varphi)=\lim_{N\to \infty}\sum_{k=1}^N \int_{\R^n} (\phi_k*u)\varphi\,dx 
\end{split}
\end{equation}exists. Now, by Cauchy's and Holder's  inequality
\begin{equation}\label{bound N}
\begin{split}
\sum_{k=1}^N \bigg|\int_{\R^n} (\phi_k*u)\varphi\,dx\bigg|  &\leq  \sum_{k=1}^N \int_{\R^n} |2^{sk} \phi_k*u||2^{-sk}\varphi|\,dx \\&\leq \int_{\R^n}\left(\sum_{k=1}^N  |2^{sk}\phi_k*u|^q\right)^{1/q}\left(\sum_{k=1}^N |2^{-sk}\varphi|^{q'}\right)^{1/q'}\,dx \qquad (q'=q/(q-1))\\& \leq C\int_{\R^n}\left(\sum_{k=1}^N  |2^{sk}\phi_k*u|^q\right)^{1/q}|\varphi|\,dx
\\ & \leq C \bigg\|\left(\sum_{k=1}^N  |2^{sk}\phi_k*u|^q\right)^{1/q}\bigg\|_{L^A(\R^n)}\|\varphi\|_{L^{\hat A}(\R^n)}.
\end{split}
\end{equation}Since the latter term in \eqref{bound N} can be bounded independently of $N$ because $u\in F_s^{A,q}(\R^n)$, the convergence in \eqref{desired conv} holds. 

Now, since the Fourier transform is continuous in $\mathcal{S}'$ and is one-to-one, from
$$\mathcal{F}\left(\Phi*u + \sum_{k=1}^\infty \phi_k*u\right) = \hat\Phi\hat u + \sum_{k=1}^{\infty}\hat \phi_k \hat u= \hat u$$we conclude \eqref{representation}.
\end{proof}

We shall decompose further this representation for $s>0$. Given $i\in \mathbb{N}$ and $k\in \mathbb{Z}^n$, let $x_{ik}=2^{-i}k$ and we define the dyadic cuber $Q_{ik}$ by setting
$$x\in Q_{ik} \text{ if and only if }x=x_{ik}+2^{-i}y, \text{ for }0\leq y_j<1,\,j=1, ..., n.$$We denote by $\chi_{ik}$ the characteristic function of the dyadic cube $Q_{ik}$. 

Given $m$ a nonnegative integer, we say that a function $a_{ik}\in C^m(\R^n)$ is a $C^m$-atom for the cube $Q_{ik}$ if it satisfies supp $a_{ik} \subset \tilde{Q}_{ik}$, where $\tilde{Q}_{ik}$ is the cube concentric with $Q_{ik}$ with three times its side, and 
\begin{equation}
\sup_{x\in \R^n, |\gamma|\leq m}2^{-i|\gamma|}|D^{\gamma}a_{ik}(x)|\leq 1.
\end{equation}

We have the following dyadic representation in $F_s^{A, 2}(\R^n)$. 
\begin{theorem}\label{atomic decomposition}
Let $u\in F_s^{A, q}(\R^n)$, with $s>0$ and suppose that $A$ and $\hat A$ satisfy the $\Delta_2$-condition. Let $m\in \mathbb{N}$ be given. Then, there exist $C^m$-atoms $\left\lbrace a_{ik}\right\rbrace$ for the dyadic cubes $Q_{ik}$ and constants $s_{ik}$ such that
\begin{equation}\label{first dec}
u(x)=\sum_{i=0}^\infty 2^{-is}u_i(x),\quad \text{with }\,u_i(x)=\sum_{k\in \mathbb{Z}^n}s_{ik}a_{ik}(x),
\end{equation}and denoting $s_{i}(x)=\sum_{k\in \mathbb{Z}^n}s_{ik}\chi_{ik}(x)$, then there is $A>0$ such that
\begin{equation}\label{estimeta l}
\|\left\lbrace s_i \right\rbrace_{i=0}^\infty\|_{L^A(\ell^q)} =\bigg\|\left(\sum_{x\in Q_{ik}}|s_{ik}|^q \right)^{1/q} \bigg\|_{L^A(\R^n)}\leq A\|u\|_{L^{A}(\R^n)}.
\end{equation}
\end{theorem}
\begin{proof}We mainly follow the proof of Theorem 4.6.2 in \cite{AH}. Let $u\in F_s^{A, q}(\R^n)$ and write \eqref{representation} as 
$$u=\sum_{i=0}^\infty 2^{-is}u_i.$$We set $\eta\in C_0^{\infty}(Q_{00})$ with $\int\eta\,dx=1$, $\eta_i(x)=2^{in}\eta(2^ix)$, and define $\eta_{ik}=\eta_i*\chi_{ik}$. Then supp $\eta_{ik}\subset \tilde{Q}_{ik}$ and $\sum_{k\in \mathbb{Z}^n}\eta_{ik}=1$. Let
$$b_{ik}=\eta_{ik}u_i, \quad s_{ik}=\max_{x\in \tilde{Q}_{ik}, |\gamma|\leq m}2^{-i|\gamma|}|D^\gamma b_{ik}(x)|$$and
$$a_{ik}(x)=\dfrac{b_{ik}(x)}{s_{ik}}.$$Observe that $a_{ik}$ are $C^m$-atoms and since
$$u_i=\sum_{k\in \mathbb{Z}^n}\eta_{ik}u_i=\sum_{k\in \mathbb{Z}^n}b_{ik}= \sum_{k\in \mathbb{Z}^n}s_{ik}a_{ik},$$we get from \eqref{representation} that
$$u=\sum_{i=0}^\infty 2^{-is}\sum_{k\in \mathbb{Z}^n}s_{ik}a_{ik}.$$This proves \eqref{first dec}. Moreover, since $\hat \Phi_{n+1}=1$ on supp $\hat u_i$, we get since Fourier transform is a bijection, that
$$u_i=\Phi_{i+1}*u_i.$$Thus, by \cite[Lemma 4.3.7]{AH}, for any $x\in \tilde{Q}_{ik}$ and $|\gamma|\leq m$, we get as in \cite[Theorem 4.6.2]{AH} that
\begin{equation}\label{Harnack}
\begin{split}
|D^\gamma u_n(x)|\leq |D^\gamma \Phi_{n+1}|*|u_n|(x) \leq A2^{i|\gamma|}\inf_{y\in Q_{ik}}Mu_i,
\end{split}
\end{equation}where $M$ is the Hardy-Littlewood maximal function. Moreover, by Leibniz formula for differentiation,
$$s_{ik}\leq A\max_{x\in \tilde{Q}_{ik}, |\gamma|\leq m}2^{-i|\gamma|}|D^\gamma u_i(x)|$$and hence by \eqref{Harnack} $s_{ik}\leq AMu_i(x)$ for all $x\in Q_{ik}$. In this way, we get
$$s_i(x)=\sum_{k\in \mathbb{Z}^n}s_{ik}\chi_{ik}(x)\leq AMu_i(x).$$The estimate \eqref{estimeta l} now follows from the Fefferman-Stein maximal theorem for Orlicz spaces (see \cite[Corollary 5.4.1]{HH}). This ends the proof of the theorem. 
\end{proof}

The converse of Theorem \ref{atomic decomposition} is also true.
\begin{theorem} \label{atomic decomposition.2}
Let $s>0$, $A$ a Young function so that $A$ and $\hat A$ satisfy the $\Delta_2$-condition and suppose that $m>s$ is an integer. Let $\left\lbrace a_{ik}\right\rbrace$ be a sequence of $C^m$-atoms for the cubes $\left\lbrace Q_{ik} \right\rbrace$ and let the coefficients $\left\lbrace s_{ik}\right\rbrace$ such that $\|\left\lbrace s_i\right\rbrace_{i=0}^\infty\|_{L^A(\ell^q)}<\infty$. Then, the function
$$u(x)=\sum_{i=0}^\infty 2^{-si}\sum_{k\in \mathbb{Z}^n}s_{ik}a_{ik}(x)$$belongs to $F_s^{A, q}(\R^n)$ and there is $C>0$ such that
\begin{equation}
\|u\|_{F_s^{A, q}(\R^n)} \leq C\|\left\lbrace s_i\right\rbrace_{i=0}^\infty\|_{L^A(\ell^q)}=C\bigg\| \left( \sum_{x\in Q_{ik}}|s_{ik}|^q\right)^{1/q}\bigg\|_{L^A(\R^n)}.
\end{equation}
\end{theorem}
\begin{proof}
The proof follows as in \cite[Theorem 4.6.3]{AH}, so we provide some details. Define for $j\in \mathbb{Z}$
$$c_j:=\begin{cases} 2^{-j(m-s)} \quad & \text{if }j\geq 0\\
2^{js} \quad & \text{if } j\leq -1.
\end{cases}$$

 With the help of \cite[Lemma 4.6.4 and Lemma 4.6.5]{AH}, there holds
\begin{equation}\label{eqq1}
|\Phi* u(x)|\leq C\sum_{j=0}^\infty 2^{-sj}Ms_j(x)= C\sum_{j=0}^\infty c_{-j}Ms_j(x),
\end{equation}and for $i\geq 1$,
\begin{equation}\label{eqq2}
2^{si}|\phi_i*u(x)|\leq C\sum_{j=0}^i2^{-(i-j)(m-s)}Ms_j(x)+A\sum_{j=i+1}^\infty 2^{(i-j)s}Ms_j(x)= C\sum_{j=0}^\infty c_{i-j}Ms_j(x),
\end{equation}where $M$ denotes the Hardy-Littlewood maximal function. Let $c=\{c_j\}_{j=-\infty}^\infty$ and $Ms(x)=\{Ms_j(x)\}_{j=-\infty}^\infty$ where $Ms_j(x)=0$ for $j<0$, and let its convolution
$$(c*Ms(x))(i)= \sum_{j=-\infty}^\infty c_{i-j}Ms_j(x).$$Since $m>s$, the series $\sum_{-\infty}^\infty|c_j|$ converges. Then, by \eqref{eqq1} and \eqref{eqq2},  and the Minkowski's inequality $\|c*Ms(x)\|_{\ell^q}\leq C\|c\|_{\ell^1}\|Ms(x)\|_{\ell^q}$, we get
\begin{equation*}
\begin{split}
&\bigg\|\left(\Phi*u(x)+\sum_{i=1}^\infty| 2^{is}\phi_i*u(x)|^q \right)^{1/q}\bigg\|_{L^A(\R^n)} \leq C\bigg\|\left(\sum_{i=-\infty}^\infty |(c*Ms(x))(i)|^q \right)^{1/q}\bigg\|_{L^A(\R^n)}\\ & \qquad\qquad = C\bigg\|\|c*Ms(x)\|_{\ell^q(\mathbb{Z})}\bigg\|_{L^A(\R^n)} \leq C\|c\|_{\ell^1(\mathbb{Z})}\|\|Ms(x)\|_{\ell^q(\mathbb{N})}\|_{L^A(\R^n)}.
\end{split}
\end{equation*}The conclusion follows from the Fefferman-Stein maximal theorem for Orlicz spaces (see \cite[Corollary 5.4.1]{HH}).
\end{proof}

\end{document}